\documentclass[reqno]{amsart}
\usepackage{amsmath,amsthm,amscd,amssymb,amsfonts, amsbsy}
\usepackage[usenames, dvipsnames]{color}
\usepackage{pxfonts}
\usepackage{enumerate}
\usepackage{latexsym}
\usepackage{mathrsfs}
\usepackage{xparse}
\usepackage{marginnote}
\usepackage{todonotes}

\numberwithin{equation}{section}

\theoremstyle{plain}
\newtheorem{theorem}{Theorem}[section]
\newtheorem{lemma}[theorem]{Lemma}

\newtheorem{proposition}[theorem]{Proposition}

\theoremstyle{definition}
\newtheorem{definition}[theorem]{Definition}

\theoremstyle{remark}
\newtheorem{remark}[theorem]{Remark}

\DeclareMathOperator{\dv}{div}
\DeclareMathOperator{\tr}{tr}

\providecommand{\set}[1]{\{#1\}}
\providecommand{\Set}[1]{\left\{#1\right\}}

\providecommand{\abs}[1]{\lvert#1\rvert}
\providecommand{\Abs}[1]{\left\lvert#1\right\rvert}

\providecommand{\biggabs}[1]{\biggl\lvert#1\biggr\rvert}

\providecommand{\norm}[1]{\lVert#1\rVert}
\providecommand{\Norm}[1]{\left\lVert#1\right\rVert}

\newcommand{\bR}{\mathbb{R}}

\begin{document}


\title[Fundamental solutions of parabolic equations in non-divergence form]{Estimates for fundamental solutions of parabolic equations in non-divergence form}

\author[H. Dong]{Hongjie Dong}
\address[H. Dong]{Division of Applied Mathematics, Brown University,
182 George Street, Providence, RI 02912, United States of America}
\email{Hongjie\_Dong@brown.edu}
\thanks{H. Dong was partially supported by the Simons Foundation, grant no. 709545, a Simons fellowship, grant no. 007638, and the NSF under agreement DMS-2055244.}

\author[S. Kim]{Seick Kim}
\address[S. Kim]{Department of Mathematics, Yonsei University, 50 Yonsei-Ro, Seodaemun-gu, Seoul 03722, Republic of Korea.}
\email{kimseick@yonsei.ac.kr}
\thanks{S. Kim was partially supported by the National Research Foundation of Korea under agreements NRF-2019R1A2C2002724 and NRF-2022R1A2C1003322.}

\author[S. Lee]{Sungjin Lee}
\address[S. Lee]{Department of Mathematics, Yonsei University, 50 Yonsei-ro, Seodaemun-gu, Seoul 03722, Republic of Korea}
\email{sungjinlee@yonsei.ac.kr}

\keywords{Fundamental solution; Parabolic equation in non-divergence form; Dini mean oscillation}

\begin{abstract}
We construct the fundamental solution of  second order parabolic equations in non-divergence form under the assumption that the coefficients are of Dini mean oscillation in the spatial variables. We also prove that the fundamental solution satisfies a sub-Gaussian estimate.
In the case when the coefficients are Dini continuous in the spatial variables and measurable in the time variable, we establish the Gaussian bounds for the fundamental solutions.
We present a method that works equally for second order parabolic systems in non-divergence form.
\end{abstract}

\maketitle

\section{Introduction and main results}

We consider second order parabolic operator $P$ in non-divergence form
\begin{equation*}				
Pu=\partial_t u - a^{ij}(t,x)D_{ij}u
\end{equation*}
in $\bR^{d+1}$.
Here and below, we use the summation convention over repeated indices.
We assume that the coefficients $\mathbf A=(a^{ij})$ are symmetric and satisfy the uniform parabolicity condition
\begin{equation}				\label{ellipticity}
\lambda \abs{\xi}^2 \le a^{ij}(t,x) \xi_i \xi_j \le \Lambda \abs{\xi}^2,\quad \forall \xi \in \bR^d,\,\, \forall (t,x) \in \bR^{d+1}.
\end{equation}
In this article, we are concerned with the fundamental solution of the operator $P$.
By the fundamental solution, we mean a function $\Gamma(t,x,s,y)$ formally satisfying
\[
P  \Gamma(\cdot, \cdot,s,y)=\delta_{{s,y}}(\cdot, \cdot)\;\text{ in }\;\bR^{d+1},
\]
or equivalently
\[
P\Gamma(\cdot, \cdot, s,y) =0\;\text{ in }\; (s,\infty )\times \bR^d,\quad \lim_{t\to s+}\Gamma(t,\cdot, s,y)=\delta_y(\cdot)\;\text{ on }\;\bR^d.
\]
We show that if the coefficients $\mathbf A=(a^{ij})$ are of Dini mean oscillation in $x$, then the fundamental solution $\Gamma(t,x,s,y)$ exists and satisfies certain estimates, in particular a sub-Gaussian estimate.
Moreover, if the coefficients are Dini continuous in $x$, then the fundamental solution enjoys the usual Gaussian bounds.
We emphasize that our methods are also applicable to parabolic systems of second order and this is one of the novelties of the paper.

Before we state our main theorems, let us introduce some basic definitions.
We define the parabolic distance between  $X=(t,x)$ and $Y=(s,y)$ in $\bR^{d+1}$ by
\[
\abs{X-Y}=\max \left(\abs{x-y}, \sqrt{\abs{t-s}}\right).
\]
We define the $(d+1)$-dimensional cylinders $Q_r(X)$, $Q_r^+(X)$, and $Q_r^-(X)$, by
\begin{align*}
Q_r(X)&=\set{ Y\in \bR^{d+1}: \abs{Y-X} < r}=(s-r^2, s+r^2)\times B_r(x),\\
Q_r^+(X)&=(s, s+r^2)\times B_r(x), \quad\text{and}\quad Q_r^-(X)=(s-r^2, s)\times B_r(x).
\end{align*}
For $X=(t,x) \in \bR^{d+1}$ and $r>0$, we define
\begin{equation*}			
\omega_{\mathbf A}^{\mathsf x}(r, X):= \fint_{Q_r^-(X)} \,\abs{\mathbf A(s,y)- \bar{\mathbf A}^{\mathsf x}_{x,r}(s)}\,dy ds,\;\text{ where }\;\bar{\mathbf A}^{\mathsf x}_{x,r}(s) :=\fint_{B_r(x)} \mathbf A(s,y)\,dy.
\end{equation*}
Then for a subset $Q$ of $\bR^{d+1}$, we define
\begin{equation*}
\omega_{\mathbf A}^{\mathsf x}(r, Q):=\sup\Set{\omega_{\mathbf A}^{\mathsf x}(r, X): X \in Q} \quad \text{and}\quad  \omega_{\mathbf A}^{\mathsf x}(r):=\omega_{\mathbf A}^{\mathsf x}(r, \bR^{d+1}).
\end{equation*}
We say that $\mathbf A$ is of \emph{Dini mean oscillation in $x$} over $Q$ and write $\mathbf A \in \mathsf{DMO_x}(Q)$ if $\omega_{\mathbf A}^{\mathsf x}(r, Q)$ satisfies the Dini condition
\[
\int_0^1 \frac{\omega_{\mathbf A}^{\mathsf x}(r,Q)}{r}\,dr <+\infty.
\]
The adjoint operator $P^\ast$ is given by
 \begin{equation*}				
 P^\ast u=-\partial_t u- D_{ij}(a^{ij}(t,x)u).
 \end{equation*}

We are now ready to state the main results.
\begin{theorem}			\label{thm1}
Assume that $\mathbf A=(a^{ij})$ satisfies \eqref{ellipticity} and belongs to $\mathsf{DMO_x}(\bR^{d+1})$.
Then, there exist unique fundamental solutions $\Gamma(X,Y)=\Gamma(t,x,s,y)$ and $\Gamma^\ast(X,Y)=\Gamma^\ast(t,x,s,y)$ for the operators $P$ and $P^\ast$, respectively, and they satisfy the symmetry relation
\begin{equation}		\label{symmetry}
\Gamma(t,x,s,y)=\Gamma^\ast(s,y,t,x).
\end{equation}
The fundamental solution $\Gamma$ is continuous in $\bR^{d+1}\times \bR^{d+1} \setminus \set{(X,X): X \in \bR^{d+1}}$ and
\[
\Gamma(t,x,s,y)=0 \quad\text{if }\;t<s.
\]
Also, for each $Y \in \bR^{d+1}$, $D_x\Gamma(\cdot, Y)$ and $D^2_x\Gamma(\cdot, Y)$ are continuous in $\bR^{d+1} \setminus \set{Y}$;
if $\mathbf{A}$ is continuous, then $\partial_t\Gamma(\cdot, Y)$ is continuous in $\bR^{d+1} \setminus \set{Y}$ as well.
Moreover, for any $R_0>0$, there exist constants $C=C(d, \lambda, \Lambda,\omega_{\mathbf A}^{\mathsf x}, R_0)$ such that we have
\begin{equation}						\label{eq2232w}
\abs{\Gamma(X,Y)} \le C \abs{X-Y}^{-d}
\end{equation}
for any $X$, $Y \in \bR^{d+1}$ satisfying $0<\abs{X-Y} <R_0$.
\end{theorem}

\begin{remark}
In Theorem~\ref{thm1}, in addition to \eqref{eq2232w}, we also have pointwise bounds for the derivatives of the fundamental solutions, that is,
\begin{equation*}
\abs{D_x\Gamma(X,Y)} \le C \abs{X-Y}^{-d-1},\quad
\abs{\partial_t \Gamma(X,Y)}+ \abs{D_x^2\Gamma(X,Y)} \le C \abs{X-Y}^{-d-2}
\end{equation*}
for any $X$, $Y \in \bR^{d+1}$ satisfying $0<\abs{X-Y} <R_0$.
These estimates follow directly from \eqref{eq2232w} and \cite[Theorem~3.2]{DEK21} applied to $\Gamma(\cdot, Y)$ in $Q_{R}^-(X)$ with $R=\frac12 \abs{X-Y}$.
\end{remark}

We recall that $\mathbf A$ belongs to $\mathsf{VMO_x}(\bR^{d+1})$ if and only if $\lim_{r\to 0} \omega_{\mathbf A}^{\mathsf x}(r)=0$ (see, e.g., \cite{Kr07}) and thus $\mathsf{VMO_x}(\bR^{d+1})$ contains $\mathsf{DMO_x}(\bR^{d+1})$.

\begin{theorem}[Sub-Gaussian estimate]			\label{thm2}
Assume that $\mathbf A=(a^{ij})$ satisfies \eqref{ellipticity} and belongs to $\mathsf{VMO_x}(\bR^{d+1})$.
Suppose there exists a fundamental solution $\Gamma(t,x,s,y)$ for the operator $P$, which satisfies \eqref{eq2232w}.
Then, for any $T>0$ and $\delta \in (0,1)$, there exist a constant $C=C(d, \lambda, \Lambda, \omega_{\mathbf A}^{\mathsf x}, \delta, T)$ and a universal constant $\beta>0$ such that for any $x$, $y \in \bR^{d}$ and $t$, $s \in \bR$ satisfying $0<t-s<T$ we have
\begin{equation}			\label{eq1742sun}
\abs{\Gamma(t,x,s,y)} \le \frac{C}{(t-s)^{d/2}} \exp\left\{-\beta \left(\frac{\abs{x-y}}{\sqrt{t-s}}\right)^{2-\delta}\right\}.
\end{equation}
\end{theorem}

We shall say that $\mathbf A$ is uniformly Dini continuous in $x$ over $\bR^{d+1}$ if its modulus of continuity in $x$ defined by
\begin{equation*}
\varrho_{\mathbf A}^{\mathsf x}(r):=
\sup\,\Set{\abs{\mathbf A(t,x)-\mathbf A(t,y)}: x, y \in \bR^d,\,  t\in \bR,\, \abs{x-y}\le r\,}
\end{equation*}
satisfies the Dini condition
\[
\int_0^1 \frac{\varrho_{\mathbf A}^{\textsf x}(r)}{r}\,dr <+\infty.
\]
It is clear the if $\mathbf A$ is uniformly Dini continuous in $x$ over $\bR^{d+1}$, then it is of Dini mean oscillation in $x$ over $\bR^{d+1}$.

\begin{theorem}[Gaussian estimate]			\label{thm3}
Assume that $\mathbf A=(a^{ij})$ satisfies \eqref{ellipticity} and $\mathbf A=(a^{ij})$ is uniformly Dini continuous in $x$ over $\bR^{d+1}$.
Then the fundamental solution satisfies the Gaussian bounds, that is, for any $T>0$, there exists $C=C(d,\lambda, \Lambda, T, \varrho_{\mathbf A}^{\mathsf x})$ and $\kappa=\kappa(\lambda, \Lambda)$ such that for any $x$, $y \in \bR^{d}$ and $t$, $s \in \bR$ satisfying $0<t-s<T$ we have
\begin{equation}			\label{eq1743sun}
\abs{\Gamma(t,x,s,y)} \le \frac{C}{(t-s)^{d/2}} \exp\left\{-\kappa \frac{\abs{x-y}^2}{t-s}\right\}.
\end{equation}
\end{theorem}

A few remarks are in order.
The fundamental solutions are topics in many classical books. See, e.g., \cite{Eidelman, Friedman, IO1962, LSU} and references therein.
It is well known that the fundamental solutions of second order parabolic equations in divergence form have two-sided Gaussian bounds even in the case when the coefficients are just bounded and measurable; see \cite{Aronson67}.
In contrast to parabolic equations in divergence form, the fundamental solutions of parabolic equations in non-divergence form do not necessarily have the Gaussian bounds if the coefficients do not possess some kind of regularity, although certain pointwise bounds are available in terms of so-called normalized adjoint solutions; see \cite{Esc2000}.
As a matter of fact, even if the coefficients are continuous in $t$ and $x$, the following weaker estimate may not hold:
\[
\abs{\Gamma(t,x,s,y)} \le C(t-s)^{-d/2}\quad \text{for }\,0<t-s<T.
\]
A counterexample is given in \cite{Ilin62} for the equation (in one space variable) $\partial u/\partial t = a(t,x) \partial^2 u/\partial x^2$ with a coefficient $a(t,x)$, continuous in $t$ and $x$, and satisfying  $1/2 \le a(t,x) \le 3/2$, whose fundamental solution is unbounded at any given point $x_0$ for any $t>0$. See also \cite{FK81} and \cite{Sa81} for examples of equations with continuous coefficients, whose fundamental solutions (as measure) are singular with respect to Lebesgue measure for any $t > 0$.
On the other hand, if the coefficients are of Dini mean oscillation in $x$, then the fundamental solutions have the usual Gaussian bounds; see \cite{DEK21}.
However, the proof there relies heavily on the Harnack type properties of nonnegative (adjoint) solutions and is not applicable to the systems setting.
To the best of our knowledge, the Gaussian bounds for the fundamental solutions were available to the non-scalar setting if the coefficients are continuous in $t$ and $x$, and additionally, if they are doubly Dini continuous in $x$. See \cite{PE84} and \cite{Eidelman}.

We give a brief description of the methods we use in the proofs. To show Theorem \ref{thm1}, we adapt an argument in \cite{KL21} for non-divergence form elliptic equations, by using the pointwise estimates of solutions established in \cite{DEK21}. In the proof of Theorem \ref{thm2}, we first establish an exponential decay estimate by using the $W^{1,2}_p$ estimate and an iteration argument. We then improve the exponential decay estimate to the sub-Gaussian estimate \eqref{eq1742sun} by exploiting  the semi-group property of the fundamental solution together with a delicate re-scaling argument. Finally, we modify the parametrix method of Levi \cite{Levi} to prove the Gaussian estimate \eqref{eq1743sun} in Theorem \ref{thm3}.
The main difference between Levi's original method and ours is that Levi's procedure was intended to construct the fundamental solution and thus required more restriction on the coefficients while in our approach, we construct the fundamental solution by different means and prove that it is identical with the resulting kernel produced by our modified parametrix method, which inherits the Gaussian bounds from the fundamental solutions of parabolic operators with coefficients depending only on $t$. It is also worth mentioning that in contrast to the scalar case, we are only able to get a one-sided Gaussian estimate.

Finally, the organization of the paper is as follows.
In Section~\ref{sec2}, we state some preliminary definition and lemmas.
The proofs of Theorems~\ref{thm1}, \ref{thm2}, and \ref{thm3} are given in Sections~\ref{sec3}, \ref{sec4}, and \ref{sec5}, respectively.

\section{Preliminaries}		\label{sec2}

For any domain $Q \subset \bR^{d+1}$ and $p\in [1,\infty]$, we shall denote by
$L_p(Q)$ the standard Lebesgue class.
We define the function space
\begin{equation*}
W_{p}^{1,2}(Q)=
\set{u:\,u, \, \partial_t u,\, Du,\, D^2u\in L_p(Q)},
\end{equation*}
which are equipped with norm
\begin{equation*}
\norm{u}_{W_{p}^{1,2}(Q)}=\norm{u}_{L_p(Q)}+\norm{Du}_{L_p(Q)}+\norm{D^2u}_{L_p(Q)}+\norm{\partial_t u}_{L_p(Q)}.
\end{equation*}

We deal with the adjoint problem
\begin{equation}\label{master-adj-prob}
P^\ast u=\dv^2\mathbf g +f\; \text{ in }\; (t_0,t_1)\times \bR^d, \quad u(t_1,\cdot)=0\; \text{on}\;  \bR^d,
\end{equation}
where $\mathbf g = (g^{kl})$ is a symmetric $d \times d$ matrix-valued function and $\dv^2\mathbf g=D_{kl}g^{kl}$.

\begin{definition}
Assume that  $\mathbf{g} \in L_p((t_0,t_1)\times \bR^d)$ and $f \in L_p((t_0,t_1)\times \bR^d)$, where $1<p<\infty$.
We say that $u \in L_p((t_0,t_1)\times \bR^d)$ is a solution to \eqref{master-adj-prob} if $u$ satisfies
\begin{equation}			\label{eq13.53pde}
\int_{t_0}^{t_1} \!\!\!\int_{\bR^d} u\, P v=\int_{t_0}^{t_1} \!\!\!\int_{\bR^d} f v+\tr(\mathbf{g}\,D^2 v)
\end{equation}
for any $v \in W^{1,2}_{p'}((t_0,t_1)\times \bR^d)$ satisfying $u(t_0,\cdot)=0$, where $\frac{1}{p}+\frac{1}{p'}=1$.
\end{definition}

\begin{lemma}				\label{lem01}
Assume that  $\mathbf{g} \in L_p((t_0,t_1)\times \bR^d)$ and $f \in L_p((t_0,t_1)\times \bR^d)$, where $1<p<\infty$.
Then there exists a unique solution $v$ of the adjoint problem \eqref{master-adj-prob} in $L_p((t_0,t_1)\times \bR^d)$ and it satisfies
\[
\norm{v}_{L_p((t_0,t_1)\times \bR^d)}  \le C \left(\norm{\mathbf g}_{L_p((t_0,t_1)\times \bR^d)}+\norm{f}_{L_p((t_0,t_1)\times \bR^d)}\right),
\]
where $C$ is a constant depending only on  $d, \lambda, \Lambda, p, t_0, t_1$, and $\omega^{\mathsf x}_{\mathbf A}$.
\end{lemma}
\begin{proof}
Recall that $\mathsf{DMO_x}(\bR^{d+1}) \subset \mathsf{VMO_x}(\bR^{d+1})$.
The existence and uniqueness of the solution to \eqref{master-adj-prob} is simple to derive by transposition from the unique existence of a solution $v \in W^{1,2}_{p'}((t_0,t_1)\times \bR^d)$ to the direct problem
\[
P v = g\;\text{ in }\; (t_0,t_1)\times \bR^d,\quad
v(t_0,\cdot) = 0 \;\text{ in }\; \bR^d,
\]
and the corresponding $L_{p'}$ estimates:
\begin{equation}			\label{eq1825w}
\norm{D^2v}_{L_{p'}((t_0,t_1)\times \bR^d)}+\norm{\partial_t v}_{L_{p'}((t_0,t_1)\times \bR^d)}+ \norm{v}_{L_{p'}((t_0,t_1)\times \bR^d)}\le C \norm{g}_{L_{p'}((t_0,t_1)\times \bR^d)},
\end{equation}
where $C=C(d, \lambda, \Lambda, t_0, t_1, \omega_{\mathbf A}^{\mathsf x})$.
See \cite{Kr07} and \cite{EM2016}.
\end{proof}

\begin{lemma}					\label{lem02}
Let $R_0>0$ and $\mathbf{g}=(g^{ij}) \in \mathsf{DMO_x}(Q^+_{R_0}(X_0))$.
Suppose $v$ is an $L_2$ solution of
\[
P^\ast u=\dv^2\mathbf{g}\quad \text{in }\; Q^+_{2r}(X_0),
\]
where $0<r \le \frac{1}{2}R_0$.
Then we have
\begin{equation*}
\norm{u}_{L_\infty(Q^+_r(X_0))}\le  C \left( \fint_{Q^+_{2r}(X_0)}\abs{u}+\int_0^r \frac{\omega^{\mathsf x}_\mathbf{g}(\tau, Q^+_{2r}(X_0))}{\tau}\,d\tau \right),
\end{equation*}
where $C=C(d,\lambda,\Lambda,\omega^{\mathsf x}_\mathbf{A},R_0)$.
\end{lemma}
\begin{proof}
The proof is essentially given in \cite[Theorem 3.3]{DEK21}.
cf. \cite[Appendix]{KL21}.
\end{proof}

\section{Proof of Theorem~\ref{thm1}}			\label{sec3}
By adapting the argument in \cite{KL21}, we shall first construct the fundamental solution $\Gamma^\ast(Y,X)=\Gamma^\ast(s,y,t,x)$ for the adjoint operator $P^*$ in Section~\ref{sec3.1}.
We then establish in Section~\ref{sec3.2} that
\[
\abs{\Gamma^\ast(Y,X)} \le C \abs{X-Y}^{-d}\quad\text{for all $X$, $Y$ satisfying $0<\abs{X-Y} <R_0$.}
\]
In Section~\ref{sec3.3}, we construct the fundamental solution $\Gamma(t,x,s,y)$ of the operator $P$ and show the symmetry relation \eqref{symmetry}, which in particular implies \eqref{eq2232w}.
\subsection{Construction of the adjoint fundamental solution}	\label{sec3.1}
Fix a point $X_0=(t_0,x_0)$ in $\bR^{d+1}$.
We construct fundamental solution $\Gamma^\ast(\cdot, X_0)=\Gamma^\ast(\cdot,\cdot, t_0, x_0)$ for the adjoint operator $P^\ast$ with a pole at $X_0=(t_0, x_0)$.
\begin{lemma}	         
For any $r>0$, $\Set{\bar{\mathbf A}^{\mathsf x}_{x_0,2^{-k}r}(\cdot)}_{k=0}^\infty$ converges in $L_1((t_0-r^2,t_0))$ to a function $\mathbf A_{x_0}(\cdot)$, which is symmetric and satisfies \eqref{ellipticity}. Moreover,
\begin{equation}
                    \label{eq12.32}
\fint_{Q_r^-(X_0)} \abs{\mathbf A-\mathbf A_{x_0}} \le c(d) \int_0^r \frac{\omega_{\mathbf A}^{\mathsf x}(s)}{s}\,ds.
\end{equation}
\end{lemma}
\begin{proof}
By the triangle inequality,
\begin{align*}
\sum_{k=0}^\infty \,\abs{\bar{\mathbf A}^{\mathsf x}_{x_0,2^{-k}r}(t)-\bar{\mathbf A}^{\mathsf x}_{x_0,2^{-k-1}r}(t)}\le\sum_{k=0}^\infty 2^{d}  \fint_{B_{2^{-k}r}(x_0)} \abs{\mathbf A(t,x)-\bar{\mathbf A}^{\mathsf x}_{x_0,2^{-k}r}(t)}\,dx.
\end{align*}
Therefore, by the Fubini theorem, we have
\begin{align}
							\nonumber
\sum_{k=0}^\infty \int_{t_0-r^2}^{t_0} &\abs{\bar{\mathbf A}^{\mathsf x}_{x_0,2^{-k}r}(t)-\bar{\mathbf A}^{\mathsf x}_{x_0,2^{-k-1}r}(t)}\,dt
 \le 2^{d} \sum_{k=0}^\infty\ \int_{t_0-r^2}^{t_0} \fint_{B_{2^{-k}r}(x_0)} \abs{\mathbf A-\bar{\mathbf A}^{\mathsf x}_{x_0,2^{-k}r}(t)}\,dxdt \\
							\nonumber
&\qquad\qquad\le 2^{d} \sum_{k=0}^\infty \sum_{j=0}^{2^{2k}-1}\int_{t_0-(j+1)(2^{-k}r)^2}^{t_0-j(2^{-k}r)^2}\fint_{B_{2^{-k}r}(x_0)} \abs{\mathbf A-\bar{\mathbf A}^{\mathsf x}_{x_0,2^{-k}r}(t)} \,dx dt\\
							\label{eq2321sun}
&\qquad\qquad \le 2^{d}\sum_{k=0}^\infty \sum_{j=0}^{2^{2k}-1}(2^{-k}r)^2 \omega_{\mathbf A}^{\mathsf x}(2^{-k}r)=2^{d}r^2 \sum_{k=0}^\infty \omega_{\mathbf A}^{\mathsf x}(2^{-k}r).							\end{align}
In view of the proof on \cite[p. 495]{Y.Li2016},
we have $\omega_{\mathbf A}^{\mathsf x}(t) \simeq \omega_{\mathbf A}^{\mathsf x}(s)$ when $t \simeq s$. Thus, \eqref{eq2321sun} implies
\begin{equation}
                            \label{eq12.11}
\sum_{k=0}^\infty \int_{t_0-r^2}^{t_0} \abs{\bar{\mathbf A}^{\mathsf x}_{x_0,2^{-k}r}(t)-\bar{\mathbf A}^{\mathsf x}_{x_0,2^{-k-1}r}(t)}\,dt
 \le C(d)r^2 \int_0^r \frac{\omega_{\mathbf A}^{\mathsf x}(s)}{s}\,ds<\infty.
\end{equation}
Therefore, $\{\bar{\mathbf A}^{\mathsf x}_{x_0,2^{-k}r}(\cdot)\}$ is a Cauchy sequence in $L_1((t_0-r^2,t_0))$.
Let $\mathbf A_{x_0}(\cdot)$ be the limit. Thus, from \eqref{eq12.11}, we have
\begin{equation}
                            \label{eq12.11b}
\int_{t_0-r^2}^{t_0} \abs{\bar{\mathbf A}^{\mathsf x}_{x_0,r}(t)-\mathbf A_{x_0}(t)}\,dt
 \le C(d)r^2 \int_0^r \frac{\omega_{\mathbf A}^{\mathsf x}(s)}{s}\,ds.
\end{equation}
Finally, by using the triangle inequality and \eqref{eq12.11b},
\begin{align*}
\fint_{Q_r^-(X_0)} \abs{\mathbf A-\mathbf A_{x_0}} &\le \fint_{Q_r^-(X_0)} \abs{\mathbf A-\bar{\mathbf A}^{\mathsf x}_{x_0,r}(t)}+\fint_{Q_r^-(X_0)} \abs{\bar{\mathbf A}^{\mathsf x}_{x_0,r}(t)-\mathbf A_{x_0}(t)}\\
&\le \omega_{\mathbf A}^{\mathsf x}(r)+\fint_{t_0-r^2}^{t_0} \abs{\bar{\mathbf A}^{\mathsf x}_{x_0,r}(t)-\mathbf A(t, x_0)}\,dt \le c(d) \int_0^r \frac{\omega_{\mathbf A}^{\mathsf x}(s)}{s}\,ds.\qedhere
\end{align*}
\end{proof}

\begin{remark}
By a slight modification of the proof above, it is easily seen that $\mathbf A_{x_0}$ is independent of $t_0$ and $r$. Moreover, if $\mathbf A$ is continuous in $x$, then clearly $\mathbf A_{x_0}(t)=\mathbf A(t,x_0)$ for a.e. $t$.
\end{remark}

We now consider the parabolic operator $P_0$ defined by
\begin{equation*}
P_0 u:= \partial_t u- a^{ij}_0(t)D_{ij}u= \partial_t u - \tr(\mathbf A_{x_0}(t) D^2 u).
\end{equation*}
Let $\Phi(t,x, s, y)$ be the fundamental solution for $P_0$.
It is well known that there are positive constants $C_0=C_0(d, \lambda, \Lambda)$ and $\kappa_0=\kappa_0(\lambda, \Lambda)$ such that
\begin{equation}			\label{eq1900m}
\abs{\Phi(t,x,s,y)} \le C_0 (t-s)^{-d/2} e^{-\kappa_0 \frac{\abs{x-y}^2}{t-s}}\quad\text{for }\;t>s
\end{equation}
and $\Phi(t,x,s,y) \equiv 0$ if $t<s$.
See, for instance, \cite[Chapter 9]{Friedman} and \cite[Chapter 2]{Krylov08}.
Since $\mathbf A_{x_0}$ does not depend on $x$, we also have
\begin{equation}			\label{eq1903m}
\Phi(t,x,s,y)=\Phi^\ast(s,y,t,x),
\end{equation}
where $\Phi^\ast$ is fundamental solution for the adjoint operator $P_0^\ast$ given by
\begin{equation*}
P_0^\ast u:= -\partial_t u- D_{ij}(a^{ij}_0(t) u)=-\partial_t u - \dv^2(\mathbf A_{x_0}(t) u).
\end{equation*}

Note that, if we set $v=\Gamma^\ast(\cdot,X_0)-\Phi^\ast(\cdot,X_0)$, then it would satisfy
\begin{align*}
P^\ast v
&=P^\ast\Gamma^\ast (\cdot,X_0)- P^\ast \Phi^\ast(\cdot,X_0) + P_0^\ast\Phi^\ast(\cdot,X_0)-P_0^\ast\Phi^\ast(\cdot,X_0)\\
&=\dv^2((\mathbf A- \mathbf A_{x_0})\Phi^\ast(\cdot, X_0)).
\end{align*}
We are thus lead to consider the problem
\begin{equation}				\label{eq1421sun}
P^\ast v=\dv^2\mathbf{g}  \;\text{ in }\; (t_0-T,t_0)\times\bR^d,\quad
v(t_0,\cdot)=0 \; \text{ on }\;\bR^d,
\end{equation}
where $T>0$ and
\begin{equation*}				
\mathbf g:=(\mathbf A- \mathbf A_{x_0})\Phi^\ast(\cdot, X_0).
\end{equation*}
By a straightforward computation using \eqref{eq1900m} and \eqref{eq1903m}, for $p \in (0, \frac{d+2}{d})$, we have
\begin{align*}
\int_{t_0-T}^{t_0}\int_{\bR^d} \abs{\mathbf g}^p\,dxdt& \le \norm{\mathbf A-\mathbf A_{x_0}}_\infty^p \int_0^T \int_{\bR^d}\left(C_0 t^{-d/2} e^{-\kappa_0 \abs{x}^2/t}\right)^p\,dxdt\\
&=\norm{\mathbf A-\mathbf A_{x_0}}_\infty^p  C_0^p \int_0^T t^{d/2-dp/2}\int_{\bR^d} t^{-d/2} e^{-\kappa_0 p \abs{x}^2/t}\,dx dt\\
&= \norm{\mathbf A-\mathbf A_{x_0}}_\infty^p  C_0^p \int_{\bR^d} e^{-\kappa_0 p \abs{x}^2}\,dx \int_0^T t^{d/2-dp/2}\,dt= C T^{\frac{2+d-dp}{2}},
\end{align*}
where $C=C(d,\lambda, \Lambda, p)$.
We just proved that
\begin{equation}				\label{eq0854m}
\mathbf g \in L_p((t_0-T, t_0)\times \bR^d), \quad \forall T>0,\;\; \forall p \in (0, \tfrac{d+2}{d}).
\end{equation}
Therefore, for $1<p<\tfrac{d+2}{d}$, by Lemma \ref{lem01} there is a unique $L_p$ solution $v$ of the problem \eqref{eq1421sun}.
By extending $v=0$ on $(t_0, \infty) \times \bR^d$ and letting $T \to \infty$,
we may assume that $v$ is defined on the entire $\bR^{d+1}$.

\begin{lemma}				
Let $v$ be as above.
The function $\Gamma^\ast(\cdot ,X_0)$ defined by
\begin{equation*}				
\Gamma^\ast(\cdot ,X_0)= v+\Phi^\ast(\cdot, X_0)
\end{equation*}
is the fundamental solution of $P^\ast$ with a pole at $X_0=(t_0, x_0)$.
\end{lemma}
\begin{proof}
For any $f\in C^\infty_c(\bR^{d+1})$, fix a $T> \abs{t_0}$ such that $(-T, T)\times \bR^d$ contains the support of $f$.
For $p'>\frac{d+2}{2}$, let $u\in W^{1,2}_{p'}((-T, T) \times \bR^d)$ be the solution of the problem $Pu=f$ with  $u(-T,\cdot)=0$.
Then, by \eqref{eq1421sun} we have
\begin{align*}
\int_{-T}^{T} \int_{\bR^{d}} vf
&=\int_{-T}^{T} \int_{\bR^{d}} vPu
=\int_{-T}^{T} \int_{\bR^{d}} (P^*v)u\\
&=\int_{-T}^{T} \int_{\bR^{d}} \tr((\mathbf A- \mathbf A_{x_0})\Phi^\ast(\cdot, X_0) D^2u) = \int_{-T}^{T} \int_{\bR^{d}}\Phi^\ast(\cdot, X_0)(P_0 u- P u)\\
&=\int_{-T}^{T} \int_{\bR^{d}} \Phi^\ast(\cdot, X_0) P_0 u -\int_{-T}^{T} \int_{\bR^{d}}\Phi^\ast(\cdot,X_0)f= u(X_0)- \int_{-T}^{T} \int_{\bR^{d}} \Phi^\ast(\cdot,X_0)f,
\end{align*}
where in the last equality we use the fact that $\Phi^\ast(\cdot, X_0)$ is the fundamental solution for $P_0^\ast$.
Therefore, we have
\begin{equation}			\label{eq2037m}
u(X_0)=\int_{-T}^{T} \int_{\bR^{d}} \Gamma^\ast(\cdot,X_0)f= \int_{\bR^{d+1}} \Gamma^\ast(\cdot,X_0)f,
\end{equation}
which means that $\Gamma^\ast(\cdot,X_0)$ is the fundamental solution for $P^\ast$ with a pole at $X_0$.
\end{proof}
Noting that $\Gamma^\ast(s,y,t,s)=0$ for $t<s$, we actually proved the following.
\begin{proposition}			\label{prop3.5}
For $p>\frac{d+2}{2}$ and $f\in L_p((t_0, t_1)\times \bR^d)$, if $u\in W^{1,2}_{p}((t_0, t_1)\times \bR^d)$ is the solution of $Pu=f$ in $(t_0, t_1)\times \bR^d$ satisfying $u(t_0,\cdot)=0$, then we have the representation formula
\[
u(t,x)= \int_{t_0}^{t} \int_{\bR^{d}} \Gamma^\ast(s,y,t,x) f(s,y)\, dy ds.
\]
\end{proposition}

\subsection{Pointwise bound for the adjoint fundamental solution}		\label{sec3.2}
Let $R_0>0$ be fixed but arbitrary.
We shall show that there exists a constant $C$ depending $R_0$ as well as on $d$, $\lambda$, $\Lambda$, and $\omega_{\mathbf A}^{\mathsf x}$ such that we have
\begin{equation}				\label{eq1722sun}
\abs{\Gamma^\ast(X,X_0)} \le C \abs{X-X_0}^{-d}\quad\text{for all $X$ satisfying $0<\abs{X-X_0} <R_0$.}
\end{equation}
Define $\mathbf{g}_1$ and $\mathbf{g}_2$ by setting
\begin{equation}					\label{eq1654t}
\mathbf{g}_1= \zeta(\mathbf{A}- \mathbf{A}_{x_0}) \Phi^\ast(\cdot, X_0)\quad\text{and}\quad \mathbf{g}_2= (1-\zeta)(\mathbf{A}-\mathbf{A}_{x_0}) \Phi^\ast(\cdot, X_0),
\end{equation}
where $\zeta$ is a smooth function on $\bR^{d+1}$ such that
\begin{equation*}			
0\leq \zeta \leq 1, \quad \zeta=0 \;\text{ in }\;Q_R(X_0),\quad \zeta=1\;\text{ in }\bR^{d+1}\setminus Q_{2R}(X_0),\quad \abs{D_x\zeta} \le 4/R,
\end{equation*}
and $R>0$ is a constant to be fixed later.
Since $\Phi^\ast(\cdot, X_0)$ vanished on $(t_0, \infty)\times \bR^d$, we see that
\begin{equation}			\label{eq1947sun}
\mathbf g_1 = \mathbf g_2 \equiv 0\quad\text{on }\; (t_0,\infty) \times \bR^d.
\end{equation}
Also, by \eqref{eq1900m} and \eqref{eq1903m},  there is a postitive constant $C_0'=C_0'(d, \lambda, \Lambda)$ such that
\begin{equation}			\label{eq1948sun}
\abs{\Phi^\ast(X, X_0)} \le C_0' \abs{X-X_0}^{-d},\quad \forall X \neq X_0.
\end{equation}
In the following lemmas, we show that $\mathbf{g}_1 \in L_{p_1}$ for  $p_1>\frac{d+2}{d}$ and $\mathbf{g}_2 \in L_{p_2}$ for $1 \le p_2 <\frac{d+2}{d}$.

\begin{lemma}					
For $p>\frac{d+2}{d}$, there is a constant $C=C(d,\lambda, \Lambda, p)$ such that
\begin{equation}				\label{eq0831m}
\norm{\mathbf g_1}_{L_{p}(\bR^{d+1})} \le C R^{\frac{d+2}{p}-d}.
\end{equation}
For $1 \le p < \frac{d+2}{d}$, there is a constant $C=C(d,\lambda, \Lambda, p)$ such that
\begin{equation}				\label{eq1540th}
\norm{\mathbf{g}_2}_{L_{p}(\bR^{d+1})} \le C\left(\int_0^{2R} \frac{\omega_{\mathbf A}^{\mathsf x}(s)}{s}\,ds\right)^{1/p} R^{\frac{d+2}{p}-d}.
\end{equation}
\end{lemma}
\begin{proof}
Note that $\norm{\mathbf A- \mathbf A_{x_0}}_\infty \le C(d, \Lambda)$.
Therefore, if $p>\frac{d+2}{d}$, we get from \eqref{eq1948sun} that
\begin{align*}				
\int_{\bR^{d+1}} \abs{\mathbf{g}_1}^{p} &\le C \sum_{k=0}^\infty \int_{Q_{2^{k+1}R}(X_0) \setminus Q_{2^k R}(X_0)} \abs{X-X_0}^{-d p}\,dX\\
&\le C \sum_{k=0}^\infty   (2^k R)^{-dp} (2^{k+1}R)^{d+2} \le C R^{d+2-dp}.
\end{align*}
When $1 \le p < \frac{d+2}{d}$,
by \eqref{eq1947sun}, \eqref{eq1948sun}, the properties of $\zeta$, and \eqref{eq12.32}, we  have
\begin{align*}				
\int_{\bR^{d+1}} \abs{\mathbf{g}_2}^{p}
&\le C \sum_{k=0}^\infty \int_{Q_{2^{1-k}R}^-(X_0) \setminus Q_{2^{-k} R}^-(X_0)}\abs{\mathbf A- \mathbf A_{x_0}}^p\, \abs{X-X_0}^{-d p}\,dX\\
&\le C \sum_{k=0}^\infty (2^{-k} R)^{-dp}\int_{Q_{2^{1-k}R}^{-}(X_0)}\abs{\mathbf A- \mathbf A_{x_0}}\\
&\le C \sum_{k=0}^\infty   (2^{-k} R)^{-dp} (2^{1-k} R)^{d+2}  \int_0^{2^{1-k}R} \frac{\omega_{\mathbf A}^{\mathsf x}(s)}{s}\,ds\\
&\le C \sum_{k=0}^\infty   (2^{-k} R)^{d+2-dp} \int_0^{2R} \frac{\omega_{\mathbf A}^{\mathsf x}(s)}{s}\,ds  \le C R^{d+2-dp}\int_0^{2R} \frac{\omega_{\mathbf A}^{\mathsf x}(s)}{s}\,ds
\end{align*}
and the lemma follows.
\end{proof}

Let $v$ be the solution of the problem \eqref{eq1421sun}.
Fix $p_1 \in (\frac{d+2}{d},\infty)$ and $p_2 \in (1, \frac{d+2}{d})$ and let  $v_i \in L_{p_i}((t_0-R_0^2,t_0)\times\bR^d)$ be the solution of the problems
\begin{equation*}			
P^\ast v_i=\dv^2\mathbf g_i  \;\text{ in }\; (t_0-R_0^2,t_0)\times\bR^d,\quad
v_i(t_0,\cdot)=0  \;\text{ on }\;\bR^d. \qquad  (i=1,2).
\end{equation*}
Then by Lemma~\ref{lem01} together with \eqref{eq0831m} and \eqref{eq1540th},  respectively, we have
\begin{equation}				\label{eq1541th}
\norm{v_1}_{L_{p_1}((t_0-R_0^2,t_0)\times\bR^d)} \le C R^{(d+2)/p_1-d}
\end{equation}
and
\begin{equation}				\label{eq1618th}
\norm{v_2}_{L_{p_2}(t_0-R_0^2,t_0)\times\bR^d)} \le  C\left(\int_0^{2R} \frac{\omega_{\mathbf A}^{\mathsf x}(r)}{r}\,dr\right)^{1/p_2} R^{(d+2)/p_2-d}.
\end{equation}
We note that the constant $C$ in the above depends on $R_0$ as well as on $d$, $\lambda$, $\Lambda$, $p_1$, $p_2$, and $\omega_{\mathbf A}^{\mathsf x}$.
By the same computation as in \eqref{eq0854m}, we find $\mathbf g_1 \in L_{p_2}((t_0-R_0^2,t_0)\times\bR^d)$ as well, and thus $v_1 \in L_{p_2}((t_0-R_0^2,t_0)\times\bR^d)$.
Therefore, by the uniqueness, we see that
\begin{equation*}		
v=v_1+v_2.
\end{equation*}
We extend $v_1$ and $v_2$ by zero on $(t_0,\infty) \times \bR^d$.

Now, for any fixed $Y_0=(s_0, y_0)$ with $0<\abs{Y_0-X_0}<R_0$, we take
\[
R=\tfrac15 \abs{Y_0-X_0}
\]
and estimate $v_1(Y_0)$ and $v_2(Y_0)$ by using Lemma~\ref{lem02} as follows:
\begin{equation}		\label{eq2017th}
\abs{v_i(Y_0)} \le  C \fint_{Q_{2R}^+(Y_0)} \abs{v_i} + C \int_0^{R} \frac{\omega_{\mathbf g_i}^{\mathsf x}(r, Q_{2R}^+(Y_0))}{r}\,dr. \quad(i=1,2).
\end{equation}
By H\"older's inequality, \eqref{eq1541th}, and \eqref{eq1618th}, we have
\begin{equation}		\label{eq2018th}
\begin{aligned}
\fint_{Q_{2R}^+(Y_0)} \abs{v_1} & \le  C R^{-(d+2)/p_1} \norm{v_1}_{L_{p_1}} \le  C R^{-d},\\
\fint_{Q_{2R}^+(Y_0)} \abs{v_2} &\le C R^{-(d+2)/p_2} \norm{v_2}_{L_{p_2}} \le C R^{-d} \left(\int_0^{2R} \frac{\omega_{\mathbf A}^{\mathsf x}(r)}{r}\,dr\right)^{1/p_2}.
\end{aligned}
\end{equation}

\begin{lemma}			\label{lem3.13}
Suppose $R:=\frac15 \abs{Y_0-X_0} >0$ and let $\eta$ be a Lipschitz function on $\bR^{d+1}$ such that $0 \le \eta \le 1$ and $\abs{D_x \eta} \le 4/R$.
Set
\[
\mathbf{g}=\eta (\mathbf{A}-\mathbf{A}_{x_0}) \Phi^\ast(\cdot, X_0).
\]
Then, for any $r \in (0, R]$ we have
\[
\omega_{\mathbf{g}}^{\mathsf x}(r, Q_{2R}^+(Y_0)) \le C R^{-d}\left(\omega_{\mathbf A}^{\mathsf x}(r)+ \frac{r}{R} \int_0^{r} \frac{\omega_{\mathbf A}^{\mathsf x}(s)}{s}\,ds\right),
\]
where $C=C(d, \lambda, \Lambda)$.
\end{lemma}
\begin{proof}
Let us denote
\[
\Phi_0(X)=\Phi_0(t,x)=\Phi^\ast(t,x,t_0,x_0)=\Phi^\ast(X,X_0).
\]
For $Z=(\tau, \xi) \in Q_{2R}^+(Y_0)$ and $0<r \le R$, we have
\begin{align*}
\fint_{Q_r^-(Z)} \abs{\mathbf g- \bar{\mathbf g}^{\mathsf x}_{\xi, r}}&=\fint_{Q_r^-(Z)} \,\Abs{\eta (\mathbf{A}-\mathbf{A}_{x_0}) \Phi_0 - \overline{(\eta(\mathbf{A}-\mathbf{A}_{x_0}) \Phi_0)}^{\mathsf x}_{\xi, r}}\\
&\le \fint_{Q_r^-(Z)} \,\Abs{(\mathbf{A}-\mathbf{A}_{x_0}) \eta\Phi_0 - \overline{(\mathbf{A}-\mathbf{A}_{x_0})}^{\mathsf x}_{\xi,r} \,\eta \Phi_0}\\
&\qquad \quad+ \fint_{Q_r^-(Z)} \,\Abs{\overline{(\mathbf{A}-\mathbf{A}_{x_0})}^{\mathsf x}_{\xi,r} \,\eta \Phi_0 - \overline{\left((\mathbf{A}-\mathbf A_{x_0}) \eta \Phi_0 \right)}^{\mathsf x}_{\xi,r}}\\
&=:I+II.
\end{align*}
Note that by the triangle inequality, $\abs{X-X_0} \ge 2R$ for any $X \in Q_r^-(Z)$ and thus, we have
\begin{equation}				\label{eq1225th}
\abs{\Phi_0(X)} + R\abs{D_x\Phi_0(X)} \le C R^{-d}, \quad \forall X \in Q_r^-(Z),
\end{equation}
where $C=C(d,\lambda, \Lambda)$.
Here the bound of $D_x\Phi_0(X)$ is due to the fact that $\mathbf{A}_{x_0}$ only depends on $t$.  See \eqref{eq1948sun} and \eqref{eq1638f}.
Therefore, we have
\begin{align}				\nonumber
I &\le \fint_{Q_r^-(Z)} \,\Abs{(\mathbf{A}-\mathbf{A}_{x_0})- \overline{(\mathbf{A}-\mathbf{A}_{x_0})}^{\mathsf x}_{\xi, r}} \abs{\Phi_0}\\
						\label{eq2222th}
&\le \fint_{Q_r^-(Z)} \,C R^{-d} \Abs{\mathbf{A}-\bar{\mathbf A}_{\xi,r}^{\mathsf x}} \le C R^{-d} \omega_{\mathbf A}^{\mathsf x}(r).
\end{align}
Also, we have
\begin{align}				\nonumber
II &= \fint_{Q_r^-(Z)} \,\Abs{\fint_{B_r(\xi)}(\mathbf{A}(t,y)-\mathbf{A}_{x_0}(t)) (\eta(t,x)\Phi_0(t,x)- \eta(t,y)\Phi_0(t,y))\,dy} dxdt \\
						\label{eq2223th}
&\le \fint_{Q_r^-(Z)} \fint_{B_r(\xi)} \abs{\mathbf{A}(t,y)-\mathbf{A}_{x_0}(t)}\, \abs{\eta(t,x)\Phi_0(t,x)- \eta(t,y)\Phi_0(t,y)}\, dy dx dt.
\end{align}
By using \eqref{eq1225th}, and the properties of $\eta$, for $(t,x) \in Q_r^-(Z)$ and $y \in B_r(\xi)$, we have
\begin{align}				\nonumber
&\abs{\eta(t,x)\Phi_0(t, x) - \eta(t,y)\Phi_0(t, y)}\\
						\nonumber
& \le \abs{\eta(t,x)} \,\abs{\Phi_0(t,x)- \Phi_0(t,y)} + \abs{\eta(t,x)- \eta(t,y)} \, \abs{\Phi_0(t,y)}\\
						\label{eq2224th}						
&\le C r R^{-d-1}+C (r/R) R^{-d} \le C r R^{-d-1}.
\end{align}
Plugging \eqref{eq2224th} into \eqref{eq2223th}, we obtain
\[
II \le C r R^{-d-1} \fint_{Q_r^-(Z)}  \abs{\mathbf{A}(t,y)-\mathbf{A}_{x_0}(t)}\, dydt.
\]
We claim that
\begin{equation}			\label{eq1307th}
\fint_{Q_r^-(Z)}  \abs{\mathbf A(t, x)-\mathbf{A}_{x_0}(t)}\, dx dt \le C \left( \frac{R \omega_{\mathbf A}^{\mathsf x}(r)}{r}+\int_0^r \frac{\omega_{\mathbf A}^{\mathsf x}(s)}{s}\,ds\right),
\end{equation}
where $C=C(d, \lambda, \Lambda)$.
Assume the claim for now.
Then, we have
\begin{equation}
						\label{eq2226th}						
II \le C r R^{-d-1} \left( \frac{R \omega_{\mathbf A}^{\mathsf x}(r)}{r}+\int_0^r \frac{\omega_{\mathbf A}^{\mathsf x}(s)}{s}\, ds\right).
\end{equation}
Combining \eqref{eq2222th} and \eqref{eq2226th}, we have (recall $r\le R$)
\begin{equation*}			
\omega_{\mathbf g}^{\mathsf x}(r, Z) \le I+II \le C R^{-d}\left(\omega_{\mathbf A}^{\mathsf x}(r)+ \frac{r}{R} \int_0^r \frac{\omega_{\mathbf A}^{\mathsf x}(s)}{s}\,ds\right).
\end{equation*}
The lemma is proved by taking supremum over $Z \in Q_{2R}^+(Y_0)$.

It remains to prove the claim \eqref{eq1307th}.
Note that we can choose a sequence of points $x_1$, $x_2$, $\ldots$, $x_N$ in $\bR^d$ with $x_N=\xi$ so that $\abs{x_{i-1}-x_i} \le r$ for $i=1,\ldots, N$ and
\begin{equation}				\label{eq2202th}
N= \lceil 7R/r \rceil  \le 8R/r.
\end{equation}
Then by using the triangle inequality, we have
\begin{equation}			\label{eq1712th}
\abs{\mathbf A(t,x)-\mathbf A_{x_0}(t)} \le \abs{\mathbf{A}(t,x)-\bar{\mathbf A}_{\xi,r}^{\mathsf x}(t)} + \sum_{i=1}^N \,\abs{\bar{\mathbf A}_{x_i,r}^{\mathsf x}(t)-\bar{\mathbf A}_{x_{i-1},r}^{\mathsf x}(t)} + \abs{\bar{\mathbf A}_{x_0,r}^{\mathsf x}(t)-\mathbf A_{x_0}(t)}.
\end{equation}
Note that by \eqref{eq12.11b}, we have
\begin{equation}			\label{eq1713th}
\fint_{\tau-r^2}^\tau \abs{\bar{\mathbf A}_{x_0,r}^{\mathsf x}(t)-\mathbf{A}_{x_0}(t)}\,dt \le c(d) \int_0^r\frac{\omega_{\mathbf A}^{\mathsf x}(s)}{s}\,ds.
\end{equation}
Also, by averaging the following triangle inequality
\[
\abs{\bar{\mathbf A}_{x_i,r}^{\mathsf x}(t)-\bar{\mathbf A}_{x_{i-1},r}^{\mathsf x}(t)} \le \abs{ \mathbf A(t,x)-\bar{\mathbf A}_{x_{i},r}^{\mathsf x}(t)} + \abs{\mathbf A(t,x)-\bar{\mathbf A}_{x_{i-1},r}^{\mathsf x}(t)}
\]
over $x \in B_r(x_{i-1})\cap B_r(x_i)$ and using $\abs{x_{i-1}-x_i} \le r$, we find that
\[
\abs{\bar{\mathbf A}_{x_i,r}^{\mathsf x}(t)-\bar{\mathbf A}_{x_{i-1},r}^{\mathsf x}(t)} \le c(d) \left( \fint_{B_r(x_{i})}  \abs{ \mathbf A(t,x)-\bar{\mathbf A}_{x_{i},r}^{\mathsf x}(t)}\,dx + \fint_{B_r(x_{i-1})}  \abs{\mathbf A(t,x)-\bar{\mathbf A}_{x_{i-1},r}^{\mathsf x}(t)}\,dx\right).
\]
Then, by averaging the last inequality over $t \in (\tau-r^2, \tau)$, we get
\begin{equation}			\label{eq1714th}
\fint_{\tau-r^2}^\tau \abs{\bar{\mathbf A}_{x_i,r}^{\mathsf x}(t)-\bar{\mathbf A}_{x_{i-1},r}^{\mathsf x}(t)} \,dt \le c(d) \omega_{\mathbf A}^{\mathsf x}(r),  \quad i=1,\ldots, N.
\end{equation}
Finally, averaging the inequality \eqref{eq1712th} over $X=(t,x) \in Q_r^-(Z)$ and using \eqref{eq1713th}, \eqref{eq1714th}, and \eqref{eq2202th}, we obtain
\[
\fint_{Q_r^-(Z)}  \abs{\mathbf A(t,x)-\mathbf A_{x_0}(t)}\,dxdt \le \omega_{\mathbf A}^{\mathsf x}(r)+ c(d) \frac{8R}{r} \omega_{\mathbf A}^{\mathsf x}(r) + c(d) \int_0^r \frac{\omega_{\mathbf A}^{\mathsf x}(s)}{s}\,ds,
\]
from which \eqref{eq1307th} follows.
\end{proof}

Applying Lemma~\ref{lem3.13} with $\eta=\zeta$ and $\eta=1-\zeta$, respectively, we get
\begin{align}				
						\nonumber
\int_0^{R} \frac{\omega_{\mathbf g_i}^{\mathsf x}(r, Q_{2R}^+(Y_0))}{r}\,dr &\le C R^{-d} \left(\int_0^{R} \frac{\omega_{\mathbf A}^{\textsf x}(r)}{r}\,dr+\frac{1}{R} \int_0^{R}\int_0^r \frac{\omega_{\mathbf A}^{\textsf x}(s)}{s}\,ds\,dr \right)\\
						\label{eq0746f}
&\le C R^{-d} \int_0^R \frac{\omega_{\mathbf A}^{\textsf x}(s)}{s}\,ds.
\end{align}
Putting \eqref{eq0746f} back to \eqref{eq2017th} together with \eqref{eq2018th}, we get
\begin{align}
						\nonumber
\abs{v_1(Y_0)} +\abs{v_2(Y_0)} &\le  C R^{-d}\left(1+ \left(\int_0^{2R} \frac{\omega_{\mathbf A}^{\mathsf x}(s)}{s}\,ds\right)^{1/p_2}+\int_0^{R} \frac{\omega_{\mathbf A}^{\mathsf x}(s)}{s}\,ds \right)\\
						\label{eq0819f}
&\le C \left(1+ \int_0^{R_0} \frac{\omega_{\mathbf A}^{\mathsf x}(s)}{s}\,ds \right) R^{-d} \le C R^{-d}.
\end{align}
Therefore, by using \eqref{eq0819f} and recalling that $v=v_1+v_2$ and $R=\frac15 \abs{X_0-Y_0}$, we have
\begin{equation}			\label{eq0824f}
\abs{v(Y_0)} \le C \abs{X_0-Y_0}^{-d},
\end{equation}
where $C=C(d,\lambda, \Lambda, \omega_{\mathbf A}^{\mathsf x}, R_0)$.
Since
\[
\Gamma^\ast(Y_0,X_0)=\Phi^\ast(Y_0,X_0)+v(Y_0)
\]
and $Y_0$ satisfies $0<\abs{Y_0-X_0}<R_0$, the estimate \eqref{eq1722sun} follows from \eqref{eq0824f} and \eqref{eq1948sun}.

\subsection{Construction of fundamental solution and the symmetry relation}	\label{sec3.3}
We shall prove that the function $\Gamma(t,x,s,y)$ given by the formula \eqref{symmetry} is the fundamental solution for the operator $P$.

For $Y =(s, y) \in \bR^{d+1}$ and $\varepsilon>0$, we first construct the approximate fundamental solution $\Gamma_\varepsilon(\cdot, Y)$ by following the strategy in \cite{CDK08}.
Let $u=\Gamma_\varepsilon(\cdot, Y)$ be the solution of the problem
\begin{equation}				\label{eq1655th}
P u= \frac{1}{\abs{Q_\varepsilon^-(Y)}} \, \chi_{Q_\varepsilon^-(Y)} \;\text{ in }\; (s-\varepsilon^2,s+T)\times\bR^d,\quad
u(s-\varepsilon^2,\cdot)=0 \;\text{ on }\;\bR^d,
\end{equation}
where $T\ge 1$ is fixed but arbitrary.
By setting $\Gamma_\varepsilon(\cdot, Y)=0$ on $(-\infty, s-\varepsilon^2)\times\bR^d$ and letting $T \to \infty$, we extend the domain of $\Gamma_\varepsilon(\cdot, Y)$ to the entire $\bR^{d+1}$.

Then by Proposition~\ref{prop3.5}, we have
\begin{equation}				\label{eq1709th}
\Gamma_\varepsilon(X, Y)= \fint_{Q_\varepsilon^-(Y)} \Gamma^\ast(Z, X)\,dZ.
\end{equation}
We conclude from \eqref{eq1709th} and \eqref{eq1722sun} that for any $X$, $Y \in \bR^{d+1}$ with $0<\abs{X-Y}<R_0$, we have
\begin{equation*}				
\abs{\Gamma_\varepsilon(X,Y)} \le C \abs{X-Y}^{-d}, \quad
\forall  \varepsilon\in \big(0,  \tfrac15 \abs{X-Y}\big),
\end{equation*}
where $C$ is a constant depending only on $d$, $\lambda$, $\Lambda$, $\omega_{\mathbf A}^{\mathsf x}$, and $R_0$.

We construct fundamental solution for the operator $P$ by modifying the method in \cite{CDK08}.
Let $Y=(s,y) \in \bR^{d+1}$ be fixed.
For any $T \ge 1$, let us denote
\[
R_{T}^{d+1}= (s-T,s+T)\times \bR^d.
\]
The following two lemmas are the adaptation of Lemmas 2.13 and 2.19 in \cite{HK20} to the parabolic setting.
\begin{lemma}				\label{lem2.17}
Let $p\in (1,\infty)$.
For any $\varepsilon\in (0,1)$, we have
\begin{align}
						\label{eq1709w}
&\int_{\bR^{d+1}_T \setminus \overline Q_r(Y)} \abs{\Gamma_\varepsilon(t,x,s,y)}^{p} \,dx dt \le C r^{-pd+d+2}, \quad \forall \, r>0\quad \text{when}\ p>(d+2)/d, \\
						\label{eq1420th}
&\int_{\bR^{d+1}_T \setminus \overline Q_r(Y)}  \abs{\partial_t \Gamma_\varepsilon(t,x,s,y)}^{p} + \abs{D_x^2 \Gamma_\varepsilon(t,x,s,y)}^{p} \,dx dt \le C r^{-(d+2)(p-1)}, \quad \forall \, r>0,
\end{align}
where $C=C(d, \lambda, \Lambda, p, T, \omega_{\mathbf A}^{\mathsf x})$.
\end{lemma}

\begin{proof}
We first establish \eqref{eq1420th}.
It is enough to consider the case when $r>4\varepsilon$.
Indeed, if $r \le 4\varepsilon$, then by \eqref{eq1825w}, we have
\begin{equation*}				
\int_{\bR^{d+1}_T} \abs{\partial_t \Gamma_\varepsilon(X, Y)}^{p} +\abs{D^2 \Gamma_\varepsilon(X, Y)}^{p} \,dX \le C \varepsilon^{-(d+2)(p-1)} \le C r^{-(d+2)(p-1)}.
\end{equation*}
For $\mathbf{g} \in C^\infty_c(\bR^{d+1}_T \setminus \overline Q_r(Y))$, let $u \in L_{q}(\bR^{d+1}_T)$ be the solution of the problem
\begin{equation*}
P^\ast u=\dv^2\mathbf g\; \text{ in }\; \bR^{d+1}_T, \quad u(s+T,\cdot)=0\; \text{on}\;  \bR^d,
\end{equation*}
where $q=p/(p-1)$.
Then by \eqref{eq13.53pde} we have
\begin{equation}				\label{eq1536fr}
\fint_{Q_\varepsilon^-(Y)} u = \int_{\bR^{d+1}_T} \tr(\mathbf{g} D^2 \Gamma_\varepsilon(\cdot, Y)).
\end{equation}

Since $\mathbf{g} = 0$ in $Q_r(Y)$, we see that  $u$ is continuous on $\overline Q_{r/2}(Y)$ by \cite[Theorem~3.3]{DEK21}.
Note that if $Z \in Q_\varepsilon^-(Y)$, then $Q_{r/2}^+(Z) \subset Q_{r}(Y)$.
It follows from Lemma \ref{lem02} that
\begin{equation}			\label{eq1720m}
\norm{u}_{L_\infty(Q_{r/4}^+(Z))} \le C r^{-d-2} \norm{u}_{L_1(Q_{r/2}^+(Z))} \le C r^{-d-2} \norm{u}_{L_1(Q_r(Y))}.
\end{equation}
Therefore, by H\"older's inequality and Lemma~\ref{lem01}, we have
\[
\norm{u}_{L_\infty(Q_{\varepsilon}^-(Y))} \le C r^{-\frac{d+2}{q}} \norm{u}_{L_{q}(Q_r(Y))} \le C r^{-\frac{d+2}{q}} \norm{u}_{L_{q}(\bR^{d+1}_T)} \le C r^{-\frac{d+2}{q}} \norm{\mathbf g}_{L_{q}(\bR^{d+1}_T)}.
\]
Since $\mathbf g$ is supported in $\bR^{d+1}_T \setminus \overline Q_r(Y)$, by \eqref{eq1536fr} and the above estimate, we have
\[
\Abs{\int_{\bR^{d+1}_T \setminus \overline Q_r(Y)} \tr(\mathbf{g} D^2 \Gamma_\varepsilon(\cdot, Y))} \le C r^{-\frac{d+2}{q}} \norm{\mathbf g}_{L_{q}(\bR^{d+1}_T \setminus \overline Q_r(Y))}.
\]
Therefore, by duality, we have
\[
\int_{\bR^{d+1}_T \setminus \overline Q_r(Y)}  \abs{D_x^2 \Gamma_\varepsilon(t,x,s,y)}^{p} \,dx dt \le C r^{-(d+2)(p-1)}.
\]
Then the estimate \eqref{eq1420th} follows from the last inequality and the fact that $P\,  \Gamma_\varepsilon(\cdot,Y)= 0$ in $\bR^{d+1}\setminus \overline Q_r(Y)$.

Next, we turn to the proof of \eqref{eq1709w}.
Again, it is enough to consider the case when $r>4\varepsilon$ because by \eqref{eq1825w} and the parabolic Sobolev embedding, we have
\begin{equation*}				
\norm{\Gamma_\varepsilon(\cdot, Y)}_{L_{p}(\bR^{d+1}_T)}  \le C \norm{\Gamma_\varepsilon(\cdot, Y)}_{W^{1,2}_{p(d+2)/(d+2+2p)}(\bR^{d+1}_T)} \le  C \varepsilon^{-d+(d+2)/p} \le C r^{-d+(d+2)/p},
\end{equation*}
where in the last inequality we used the fact that $-d+(d+2)/p<0$.
For $f \in C^\infty_c(\bR^{d+1}_T \setminus \overline Q_r(Y))$, let $u \in L_q(\bR^{d+1}_T)$ be the solution of the problem
\begin{equation*}
P^\ast u=f\; \text{ in }\; \bR^{d+1}_T, \quad u(s+T,\cdot)=0\; \text{on}\;  \bR^d.
\end{equation*}
Then by \eqref{eq13.53pde} we have
\begin{equation}				\label{eq2116w}
\fint_{Q_\varepsilon^-(Y)} u = \int_{\bR^{d+1}_T} f \Gamma_\varepsilon(\cdot, Y).
\end{equation}
Similar to \eqref{eq1720m}, for $Z\in Q_\varepsilon^-(Y)$, we have
\begin{equation}			\label{eq1721m}
\norm{u}_{L_\infty(Q_{r/4}^+(Z))} \le C r^{-d-2} \norm{u}_{L_1(Q_r(Y))}.
\end{equation}
Let $v$ be the solution of
\begin{equation*}
-\partial_t v-\Delta v =f\; \text{ in }\; \bR^{d+1}_T, \quad v(s+T,\cdot)=0\; \text{on}\;  \bR^d.
\end{equation*}
By the $L_p$ estimates (cf. \eqref{eq1825w}) and the parabolic Sobolev embedding, we have
\begin{equation*}			
\norm{v}_{L_{q(d+2)/(d+2-q)}(\bR^{d+1}_T)} \le C \norm{v}_{W^{1,2}_{q}(\bR^{d+1}_T)} \le C \norm{f}_{L_{q}(\bR^{d+1}_T)}.
\end{equation*}
Note that $w=u-v$ satisfies
\begin{equation*}
P^\ast w=-\dv^2((\mathbf A-\mathbf I)v)\; \text{ in }\; \bR^{d+1}_T, \quad w(s+T,\cdot)=0\; \text{on}\;  \bR^d.
\end{equation*}
Therefore, by Lemma~\ref{lem01} and the last inequality, we have
\begin{equation*}			
\norm{w}_{L_{q(d+2)/(d+2-q)}(\bR^{d+1}_T)}
\le C \norm{\mathbf A-\mathbf I}_\infty \norm{v}_{L_{q(d+2)/(d+2-q)}(\bR^{d+1}_T)} \le C \norm{f}_{L_{q}(\bR^{d+1}_T)},
\end{equation*}
which in turn implies that
\begin{equation*}			
\norm{u}_{L_{q(d+2)/(d+2-q)}(\bR^{d+1}_T)} \le \norm{v}_{L_{q(d+2)/(d+2-q)}(\bR^{d+1}_T)} + \norm{w}_{L_{q(d+2)/(d+2-q)}(\bR^{d+1}_T)}  \le C \norm{f}_{L_{q}(\bR^{d+1}_T)}.
\end{equation*}
Then by \eqref{eq1721m} and H\"older's inequality, we have
\[
\norm{u}_{L_\infty(Q_{\varepsilon}^-(Y))} \le C r^{-d+(d+2)/p} \norm{u}_{L_{q(d+2)/(d+2-2q)}(Q_r(Y))} \le C r^{-d+(d+2)/p} \norm{f}_{L_{q}(\bR^{d+1}_T)}.
\]
Therefore, it follows from \eqref{eq2116w} and the assumption that $f=0$ in $Q_r(Y)$, that
\[
\Abs{\int_{\bR^{d+1}_T \setminus \overline Q_r(Y)} f \Gamma_\varepsilon(\cdot, Y)} \le C r^{-d+(d+2)/p} \norm{f}_{L_{q}(\bR^{d+1}_T\setminus \overline Q_r(Y))}.
\]
Again, we obtain \eqref{eq1709w} from the last inequality by duality.
\end{proof}

\begin{lemma}					\label{lem2.13}
For any $\varepsilon\in (0,1)$, we have
\begin{align*}
\Abs{\Set{(t,x) \in \bR^{d+1}_T : \abs{\Gamma_\varepsilon(t,x,s,y)}>\alpha}} &\le C \alpha^{-\frac{d+2}{d}}, \quad \forall \, \alpha>0, \\
\Abs{\Set{(t,x) \in \bR^{d+1}_T : \abs{\partial_t  \Gamma_\varepsilon(t,x,s,y)}+ \abs{D^2_x  \Gamma_\varepsilon(t,x,s,y)}>\alpha}} &\le C \alpha^{-1},  \quad \forall \, \alpha>0,
\end{align*}
where $C=C(d, \lambda, \Lambda, T, \omega_{\mathbf A}^{\mathsf x})$.
\end{lemma}
\begin{proof}
These follow from \eqref{eq1709w} and \eqref{eq1420th},  respectively.
See the proof of \cite[Lemma~3.4]{CDK08}.
\end{proof}

With Lemmas \ref{lem2.17}  and \ref{lem2.13} available, one can modify the argument of \cite{CDK08}  to construct the fundamental solution $\Gamma(X,Y)$ for the operator $P$ out of the family $\set{\Gamma_\varepsilon(X,Y)}$.
We claim that for any $p \in (1,\infty)$, $r>0$, and $T \ge 1$, we have
\begin{equation}
                \label{eq3.00}
\sup_{0<\varepsilon < 1} \norm{\Gamma_\varepsilon(\cdot, Y)}_{W^{1,2}_p(\bR^{d+1}_T \setminus \overline Q_r(Y))} <+\infty.
\end{equation}
Indeed, by using the fact that $\Gamma_\varepsilon(s-T,\cdot)\equiv 0$, it follows from the Poincar\'e inequality and Lemma~\ref{lem2.17} that
\[
\int_{s-T}^{s+T}\!\!\!\int_{\bR^d \setminus B_r(y)}
\abs{\Gamma_\varepsilon(t,x,s,y)}^p\,dxdt
\le C  \int_{s-T}^{s+T}\!\!\!\int_{\bR^d \setminus B_r(y)}\abs{\partial_t\Gamma_\varepsilon(t,x,s,y)}^p\,dx dt \le C,
\]
where $C$ is a constant that depends on the parameters including $p$, $r$, and $T$ but is independent of $\epsilon$.
Then, by the interpolation inequality, we have
\begin{align*}
\int_{s-T}^{s+T}\!\!\!\int_{\bR^d \setminus B_r(y)}&
 \abs{D_x\Gamma_\varepsilon(t,x,s,y)}^p \,dxdt \\
&\le C  \int_{s-T}^{s+T}\!\!\!\int_{\bR^d \setminus B_r(y)}  \abs{\Gamma_\varepsilon(t,x,s,y)}^p + \abs{D_x^2\Gamma_\varepsilon(t,x,s,y)}^p\,dxdt \le C.
\end{align*}
Let $\eta=\eta(x)$ be a smooth function such that
\[
0\leq \eta \leq 1, \quad \eta=1 \;\text{ in }\;B_r(y),\quad \eta=0\;\text{ in }\bR^{d}\setminus B_{2r}(y),\quad \abs{D \eta} \le 2/r.
\]
We apply the Poincar\'e inequality in the space variable to $\eta D_x\Gamma(\cdot,\cdot, s,y)$ on $I \times B_{2r}(y)$ for $I=(s-T, s-r^2)$ and $I=(s+r^2, s+T)$, separately, to get
\begin{align*}
\int_{I} \int_{B_r(y)} \abs{D_x\Gamma_\varepsilon(t,x,s,y)}^p\,dxdt
&\le C \int_I \int_{B_{2r}(y)} \abs{D_x^2 \Gamma_\varepsilon(t,x,s,y)}^p\,dxdt \\
&\quad +C  r^{-p} \int_I \int_{B_{2r}(y)\setminus B_r(y)} \abs{D_x\Gamma_\varepsilon(t,x,s,y)}^p\,dxdt \le C,
\end{align*}
and similarly with $\eta \Gamma(\cdot,\cdot, s,y)$ in place of $\eta D_x\Gamma(\cdot,\cdot, s,y)$, we get
\begin{align*}
\int_{I} \int_{B_r(y)} \abs{\Gamma_\varepsilon(t,x,s,y)}^p\,dxdt
&\le C \int_I \int_{B_{2r}(y)} \abs{D_x \Gamma_\varepsilon(t,x,s,y)}^p\,dxdt \\
&\quad +C  r^{-p} \int_I \int_{B_{2r}(y)\setminus B_r(y)} \abs{\Gamma_\varepsilon(t,x,s,y)}^p\,dxdt \le C.
\end{align*}
Combining these together, we obtain \eqref{eq3.00}.
Therefore, by applying a diagonalization process, we see that there exists a sequence of positive numbers $\set{\varepsilon_i}_{i=1}^\infty$ with $\lim_{i\to \infty} \varepsilon_i=0$ and a function $\Gamma(\cdot, Y)$ on $\bR^{d+1}\setminus \set{Y}$, which belongs to $W^{1,2}_{2}(\bR^{d+1}_T \setminus \overline Q_r(Y))$ for any $T\ge 1$ and $r>0$, such that
\begin{equation}				\label{eq1640m}
\Gamma_{\varepsilon_i}(\cdot, Y) \rightharpoonup \Gamma(\cdot, Y) \;\text{ weakly in } W^{1,2}_{2}(\bR^{d+1}_T\setminus \overline Q_r(Y)).
\end{equation}
On the other hand,  Lemma~\ref{lem2.13} implies that for $1<p<\frac{d+2}{d}$, we have
\[
\sup_{0<\varepsilon <1} \norm{\Gamma_\varepsilon(\cdot, Y)}_{L_p(\overline Q_r(Y))} <+\infty,
\]
which together with \eqref{eq3.00} implies that
\[
\sup_{0<\varepsilon <1} \norm{\Gamma_\varepsilon(\cdot, Y)}_{L_p(\bR^{d+1}_T)} <+\infty,
\]
Therefore, by passing to a subsequence if necessary, we see that
\[
\Gamma_{\varepsilon_i}(\cdot, Y) \rightharpoonup \Gamma(\cdot, Y)\;\text{ weakly in } \; L_p(\bR^{d+1}_T),\quad \forall\, p \in (1, \tfrac{d+2}{d}).
\]
Finally, from \eqref{eq1640m} and \eqref{eq1655th}, we find that $\Gamma(\cdot, Y)$ belongs to $W^{1,2}_2(\bR^{d+1}_T \setminus \overline Q_r(Y))$ and satisfies $P \,\Gamma(\cdot, Y) =0$ in $\bR^{d+1}_T \setminus \overline Q_r(Y)$.
Since we assume that $\mathbf A$ belongs to  $\mathsf{DMO_x} \subset \mathsf{VMO_x}$, we see that for any $r>0$, $\Gamma_\varepsilon(\cdot, Y)$ is locally uniformly continuous in $\bR^{d+1} \setminus Q_r(Y)$ for sufficiently small $\varepsilon$'s, with a uniform modulus of continuity.
Thus, by the Arzela-Ascoli theorem and passing to another subsequence if necessary, we see that
\[
\Gamma_{\varepsilon_i}(\cdot, Y) \to \Gamma(\cdot, Y)\;\text{ locally uniformly on }\; \bR^{d+1} \setminus  Q_r(Y),\quad  \forall\, r>0 .
\]

Recall that $\Gamma^\ast(\cdot, X)$ satisfies
\[
P^\ast \Gamma^\ast(\cdot, X)=0\;\text{ in }\;\bR^{d+1}\setminus Q_r(X)\;\text{ for any }\;r>0,
\]
and thus by \cite[Theorem 3.3]{DEK21}, we see that $\Gamma^\ast(\cdot, X)$ is continuous in $\bR^{d+1}\setminus \set{X}$.
Therefore, we obtain the identity \eqref{symmetry} by taking limit $\varepsilon \to 0$ in \eqref{eq1709th}.

Note that we have just shown that $\Gamma(X, Y)$ is continuous in $\bR^{d+1}\times \bR^{d+1}$ away from the diagonal $\set{(X,X): X \in \bR^{d+1}}$.
The property that $\Gamma(t,x,s,y)=0$ for $t<s$ follows from the fact that $\Gamma_\varepsilon(t,x,s,y)=0$ if $t \le s-\varepsilon^2$.
Also, it follows from \cite[Theorem~3.2]{DEK21} that $D^2_x\Gamma(\cdot, Y)$ is continuous in $\bR^{d+1} \setminus \set{Y}$ and that $\partial_t\Gamma(\cdot, Y)$ is continuous in $\bR^{d+1} \setminus \set{Y}$ if $\mathbf{A}$ is continuous.
We obtain \eqref{eq2232w} immediately from \eqref{eq1722sun}.\qed

\section{Proof of Theorem~\ref{thm2}}			\label{sec4}
For the sake of simplicity, let us assume that $Y=0$ and $T=1$.
Also, let us denote
\begin{equation*}				
u(t, x)=\Gamma(t,x, 0,0).
\end{equation*}
In Section~\ref{sec4.1}, we first show that $u(t,x)$ has the exponential decay
\[
\abs{u(t,x)} \le C_0 t^{-d/2} \exp(-\kappa_0 \abs{x}/\sqrt{t})
\]
for some $\kappa_0>0$ and $C_0>1$.
Then in Section~\ref{sec4.2}, by using the semigroup property
\begin{equation}				\label{eq1549th}
\Gamma(t,x,s,y)=\int_{\bR^d} \Gamma(t,x, \tau, \xi) \Gamma(\tau, \xi, s,y) \,d\xi,\quad \text{for }\;s<\tau<t,
\end{equation}
iteratively with appropriately chosen time steps, we establish the almost Gaussian estimate \eqref{eq1742sun}.
\subsection{Exponential decay of the fundamental solution}	\label{sec4.1}
For $k=1, 2,\ldots$, let $\eta_k=\eta_k(x)$ be a smooth function in $\bR^d$ such that
\[
\eta_k=0\;\text{ in }\;B_k(0),\quad \eta_k=1\;\text{ in }\;\bR^d \setminus B_{k+1}(0),\quad \norm{D\eta}_\infty \le 2, \quad \norm{D^2\eta}_\infty \le 4.
\]
Let $v=ue^{-\mu t}$, where $\mu \ge 1$ is a constant to be specified.
Note that
\[v_k=v_k(t,x) :=\eta_k(x) v(t,x)
\]
satisfies
\begin{equation*}
P v_k+\mu v_k
=f_k:= -2a^{ij}D_i \eta_k D_jv
-v a^{ij}D_{ij} \eta_k \;\text{ in }\; (0,1) \times\bR^d,\quad  v_k(0,\cdot)=0\; \text{ on }\;\bR^d.
\end{equation*}
Let us denote
\[
B_k=B_k(0),\quad B_k^c=\bR^d \setminus B_k(0).
\]
By the $W^{1,2}_p$-estimates (see, for instance, \cite{Kr07}), we have
\begin{multline*}
\mu\norm{v_k}_{L_p((0,1)\times \bR^d)} +\sqrt{\mu} \norm{D v_k}_{L_p((0,1)\times \bR^d)}
+\norm{D^2 v_k}_{L_p((0,1)\times \bR^d)}\\
 \le N_0 \norm{f_k}_{L_p((0,1)\times \bR^d)}\le N_0 \left(\norm{Dv}_{L_p((0,1)\times (B_{k+1}\setminus B_{k}))} + \norm{v}_{L_p((0,1)\times (B_{k+1}\setminus B_{k}))} \right),
\end{multline*}
where $N_0=N_0(d, \lambda, \Lambda, p, \omega_{\mathbf A}^{\mathsf x})$ is independent of $\mu$.
On the other hand, note that
\begin{equation*}
\norm{v}_{L_p((0,1)\times B_{k+1}^c)}+\norm{Dv}_{L_p((0,1)\times B_{k+1}^c)} \le \norm{v_k}_{L_p((0,1)\times \bR^d)} + \norm{Dv_k}_{L_p((0,1)\times \bR^d)}.
\end{equation*} 
Combining the last two inequalities, we have
\begin{align}
								\nonumber
\norm{v}_{L_p((0,1)\times B_{k+1}^c)}+\norm{Dv}_{L_p((0,1)\times B_{k+1}^c)} &\le N_0 \mu^{-\frac12} \left( \norm{v}_{L_p((0,1)\times (B_{k+1}\setminus B_{k}))} + \norm{Dv}_{L_p((0,1)\times (B_{k+1}\setminus B_{k}))}\right)\\
								\label{eq1222w}
&\le N_0\mu^{-\frac12} \left( \norm{v}_{L_p((0,1)\times B_{k}^c)} + \norm{Dv}_{L_p((0,1)\times B_{k}^c)}\right).
\end{align}
Taking $\mu$ so large that $N_0\mu^{-1/2}\le 1/2$ and iterating on $k=1,2,3,\ldots$ in \eqref{eq1222w}, we get
\begin{multline}				\label{eq1223w}
\norm{v}_{L_p((0,1)\times B_{k+1}^c)}+\norm{Dv}_{L_p((0,1)\times B_{k+1}^c)} \\
\le 2^{-k} \left( \norm{v}_{L_p((0,1)\times (B_2\setminus B_1))} + \norm{Dv}_{L_p((0,1)\times (B_2\setminus B_1))}\right)\le C2^{-k}
\end{multline}
for $k=1,2,3, \ldots$, where we used the local $W^{1,2}_p$ estimate and the pointwise estimate \eqref{eq2232w} in the last inequality.

Then, by using \eqref{eq1223w}, the fact that $Pu=0$ in $(0,1)\times \bR^d$, and \eqref{eq2232w} we find that there are constants $C_0 > 1$ and  $\kappa_0>0$ such that
\begin{equation}				\label{eq2312w}
\abs{u(1,x)} \le C_0 e^{-\kappa_0 \abs{x}},\quad \forall x \in \bR^d.
\end{equation}
We remark that in the proof of \eqref{eq2312w} above, we only used the bound \eqref{eq2232w} with $Y=0$.

Notice that for $\varepsilon \in (0,1]$,  if we set $\tilde u$ and $\tilde a^{ij}$ by
\[
\tilde u(t,x)=\varepsilon^d u(\varepsilon^2t, \varepsilon x),\quad \tilde a^{ij}(t,x)=a^{ij}(\varepsilon^2 t, \varepsilon x),
\]
and define the operator $\tilde P$ by
\[
\tilde P \tilde u:= \partial_t \tilde u - \tilde a^{ij} D_{ij} \tilde u,
\]
then it is easily seen that $\tilde u(t,x)$ satisfies $\tilde P \tilde u=0$ in $(0,1)\times \bR^d$ and that $\tilde u$ satisfies the bound \eqref{eq2232w} with $Y=0$, i.e.,
\[
\abs{\tilde u(t,x)} \le C \max(\sqrt{t},\abs{x})^{-d}.
\]
Since $0<\varepsilon \le 1$, we can keep the same the constants $C_0$ and $\kappa_0$ in \eqref{eq2312w} for $\tilde u$ and  obtain
\begin{equation}				\label{eq2310w}\abs{\Gamma(\varepsilon^2,x,0,0)}=\varepsilon^{-d}\abs{\varepsilon^d u(\varepsilon^2, \varepsilon x/\varepsilon )}=\varepsilon^{-d} \abs{\tilde u(1,x/\varepsilon)} \le C_0 \varepsilon^{-d} e^{-\kappa_0 \abs{x}/\varepsilon}.
\end{equation}
Also, since translation does not alter the constants $\kappa_0$ and $C_0$ in the estimate \eqref{eq2310w}, for any $x$, $y \in \bR^{d}$ and $s \in \bR$, we have
\begin{equation}				\label{eq2313w}
\abs{\Gamma(s+\varepsilon^2 ,x,s,y)}  \le C_0 \varepsilon^{-d} e^{- \kappa_0\frac{\abs{x-y}}{\varepsilon}},\quad \forall \varepsilon \in (0,1].
\end{equation}

\subsection{Almost Gaussian estimate}			\label{sec4.2}
For $(t,x) \in (0,1] \times \bR^d$, let $N=N(t,x)> 1$ be an integer to be chosen later.
We partition the interval $(0,1)$ into $N^2$ subintervals of equal length $t/N^2$.
Let us denote
\[
t_j=j(t/N^2),\quad j=1,2,\ldots, N^2.
\]
By using \eqref{eq2313w} and \eqref{eq1549th}, we have
\[
\Gamma(t_{j+1},x_{j+1},0,0)=\int_{\bR^d}\Gamma(t_{j+1},x_{j+1},t_{j},x_j)\Gamma(t_{j},x_{j},0,0)\,dx_j.
\]
Inductively, we have
\[
\Gamma(t_{N^2},x_{N^2},0,0)=\int_{\bR^d}\cdots \int_{\bR^d} \prod_{j=1}^{N^2-1} \Gamma(t_{j+1},x_{j+1},t_{j},x_{j}) \Gamma(t_{1},x_{1},0,0)\,dx_{1}\cdots dx_{N^2-1}.
\]
Therefore, by using \eqref{eq2313w} with $\varepsilon=\sqrt{t}/N$, we have
\begin{align}
							\nonumber
\abs{\Gamma(t_{N^2},x_{N^2},0,0)}
&\le \left(\frac{C_0 N^d}{t^{d/2}}\right)^{N^2}\int_{(\bR^d)^{N^2-1}} \left( \prod_{j=1}^{N^2-1}  e^{-\kappa_0\frac{N\abs{x_{j+1}-x_j}}{\sqrt{t}}}\right) e^{-\kappa_0\frac{N\abs{x_1}}{\sqrt{t}}} \,dx_{1}\cdots dx_{N^2-1}\\
							\label{eq0918sun}
&\le C_0^{N^2} \left(\frac{N^d}{t^{d/2}}\right) \int_{(\bR^d)^{N^2-1}}   e^{-\kappa_0\sum_{j=1}^{N^2-1} \abs{y_j}-\kappa_0 \Abs{\frac{N}{\sqrt{t}} x_{N^2}-\sum_{j=1}^{N^2-1} y_j}}\,dy_{1}\cdots dy_{N^2-1},
\end{align}
where we used the change of variables
\[
y_1=\frac{N}{\sqrt{t}} x_1;\qquad
y_j=\frac{N}{\sqrt{t}}(x_j-x_{j-1}),\;\; j=2,\ldots, N^2-1.
\]
By the triangle inequality, for any $(y_1, \ldots, y_{N^2-1})\in (\bR^{d})^{N^2-1}$, we have
\begin{equation}				\label{eq2228sun}
\sum_{j=1}^{N^2-1} \abs{y_j}+\biggabs{\frac{N}{\sqrt{t}} x-\sum_{j=1}^{N^2-1} y_j} \ge  \sum_{j=1}^{N^2-1} \abs{y_j} + \biggabs{\frac{N}{\sqrt{t}}x} - \biggabs{\sum_{j=1}^{N^2-1} y_j} \ge \frac{N}{\sqrt{t}}\abs{x}.
\end{equation}
For $n=0, 1,2,\ldots$, let us denote 
\begin{equation*}
\Omega_n=\Set{(y_1, \ldots, y_{N^2-1})\in (\bR^{d})^{N^2-1}: n\frac{N}{\sqrt{t}}\abs{x} \le \sum_{j=1}^{N^2} \abs{y_j} < (n+1) \frac{N}{\sqrt{t}}\abs{x}, \;\;\sum_{j=1}^{N^2}  y_j=\frac{N}{\sqrt{t}}x  }.
\end{equation*}
If $(y_1, \ldots, y_{N^2-1}) \in \Omega_n$, then we have
\begin{equation}				\label{eq2227sun}
\sum_{j=1}^{N^2-1} \abs{y_j}+\biggabs{\frac{N}{\sqrt{t}} x-\sum_{j=1}^{N^2-1} y_j}= \sum_{j=1}^{N^2} \abs{y_j} \ge n\frac{N}{\sqrt{t}}\abs{x}.
\end{equation}
Notice that $d(N^2-1)$-dimensional Lebesgue measure $\abs{\Omega_n}$ is bounded by
\begin{equation}				\label{eq2229sun}
\abs{\Omega_n} \le \left(2(n+1) \frac{N}{\sqrt{t}}\abs{x} \right)^{d(N^2-1)}
\end{equation}
and $\Omega_0=\emptyset$.

By taking $x_{N^2}=x$ and decomposing the last integral in \eqref{eq0918sun} into the sums of integrals over $\Omega_n$, we obtain from \eqref{eq2227sun}, \eqref{eq2228sun}, and \eqref{eq2229sun} that
\begin{align}
							\nonumber
&\abs{\Gamma(t,x,0,0)}
\le C_0^{N^2} \left(\frac{N^d}{t^{d/2}}\right) \left(e^{-\kappa_0  \frac{N\abs{x}}{\sqrt{t}}} \abs{\Omega_1}+ \sum_{n=2}^\infty e^{-\kappa_0 n \frac{N\abs{x}}{\sqrt{t}}}  \abs{\Omega_n} \right)\\
							\label{eq1144sun}
&\le C_0^{N^2}  \left(\frac{N^d}{t^{d/2}}\right) \left(\frac{2N\abs{x}}{\sqrt{t}} \right)^{d(N^2-1)} e^{-\kappa_0  \frac{N\abs{x}}{\sqrt{t}}} \left(2^{d(N^2-1)} + \sum_{n=2}^\infty (n+1)^{d(N^2-1)}e^{-\kappa_0 (n-1) \frac{N\abs{x}}{\sqrt{t}}}\right).
\end{align}
By the integral comparison, the binomial formula, and Stirling's formula, we have
\begin{align}
							\nonumber
\sum_{n=2}^\infty (n+1)^{k}e^{-\alpha (n-1)}
&\le \sum_{n=2}^\infty \int_{n-2}^{n-1} (s+3)^{k}e^{-\alpha s}\,ds
= \int_0^\infty (s+3)^{k}e^{-\alpha s} \,ds\\
							\nonumber
&=\int_0^\infty \sum_{m=0}^k {k \choose m}\, s^{m}3^{k-m} e^{-\alpha s} \,ds = \sum_{m=0}^k {k \choose m}\, 3^{k-m}\alpha^{-m-1}  \int_{0}^\infty s^{m}e^{-s} \,ds\\
							\nonumber
&= \sum_{m=0}^k {k \choose m}\, 3^{k-m}\alpha^{-m-1} m! \le  \frac{3^k}{\alpha}k! \sum_{m=0}^k {k \choose m}\,\left( \frac{1}{3\alpha}\right)^{m}\\
							\label{eq1145sun}
&=\frac{k!}{\alpha}\left(3+\frac{1}{\alpha}\right)^k \le c_0\frac{\sqrt{k}}{\alpha}   \left(\frac{k}{e}\right)^k \left(3+\frac{1}{\alpha}\right)^k,
\end{align}
where $c_0$ is an absolute constant.
By combining \eqref{eq1144sun} and \eqref{eq1145sun}, we have
\begin{align}
							\nonumber
\abs{\Gamma(t,x,0,0)}& \le
\frac{1}{t^{d/2}} e^{-\kappa_0 \frac{N\abs{x}}{\sqrt{t}}}  C_0^{N^2}  N^d  \left(\frac{4N\abs{x}}{\sqrt{t}} \right)^{d(N^2-1)} \\
							\nonumber
&\!\!\!\!\!\!\!\!\!\!+ c_0 \frac{1}{t^{d/2}} e^{-\kappa_0  \frac{N\abs{x}}{\sqrt{t}}}\sqrt{d(N^2-1)}\left( \frac{\sqrt{t}}{\kappa_0 N\abs{x}} \right)C_0^{N^2}  N^d  \left\{ \frac{2d(N^2-1)}{e}\left(\frac{3N\abs{x}}{\sqrt{t}} +\frac{1}{\kappa_0}\right)\right\}^{d(N^2-1)}\\
							\nonumber
& \le \frac{1}{t^{d/2}} e^{-\kappa_0 \frac{N\abs{x}}{\sqrt{t}}}  C_0^{N^2}  N^d  \left(\frac{4N\abs{x}}{\sqrt{t}} \right)^{d(N^2-1)} \\
							\label{eq1629m}
&\quad+\frac{c_0\sqrt{d}}{t^{d/2}} e^{-\kappa_0  \frac{N\abs{x}}{\sqrt{t}}}\left( \frac{\sqrt{t}}{\kappa_0 \abs{x}} \right)C_0^{N^2}  N^d \left\{ \frac{2d(N^2-1)}{e}\left(\frac{3N\abs{x}}{\sqrt{t}} +\frac{1}{\kappa_0}\right)\right\}^{d(N^2-1)}.
\end{align}
Let us write $\xi=x/\sqrt{t}$ and take $N=\lceil\abs{\xi}^{1-\delta}\rceil$, where $\delta \in (0,1)$ is fixed but arbitrary.
Note that
\[
\abs{\xi}^{1-\delta} \le N < \abs{\xi}^{1-\delta}+1.
\]
Let us consider
\begin{align*}
A&=-\kappa_0 \abs{\xi}^{2-\delta}+(\log C_0) (\abs{\xi}^{1-\delta}+1)^2+d  \log(\abs{\xi}^{1-\delta}+1)\\
&\qquad \qquad+d\{(\abs{\xi}^{1-\delta}+1)^2 -1)\log(4 (\abs{\xi}^{1-\delta}+1)\abs{\xi}),\\
B &=-\kappa_0\abs{\xi}^{2-\delta} - \log(\kappa_0 \abs{\xi}) +(\log C_0) (\abs{\xi}^{1-\delta}+1)^2+d \log(\abs{\xi}^{1-\delta}+1)\\
&\qquad \qquad+ d\{(\abs{\xi}^{1-\delta}+1)^2 -1)\log(2d((\abs{\xi}^{1-\delta}+1)^2 -1)(3\abs{\xi}(\abs{\xi}^{1-\delta}+1)+\kappa_0^{-1})e^{-1}).
\end{align*}
Note that there exist $R_0=R_0(\delta, C_0, d, \kappa_0)\ge 1$ such that if $\abs{\xi} > R_0$, then
\[
A \le - \beta \abs{\xi}^{2-\delta}, \quad
B \le -\beta \abs{\xi}^{2-\delta},
\]
where $\beta=\kappa_0/2$.
Then, it follows from \eqref{eq1629m} that for any $(t,x) \in (0,1] \times \bR^d$ with $\abs{x}/\sqrt{t} > R_0$, we have
\[
\abs{\Gamma(t,x,0,0)} \le C_1 t^{-d/2} \exp\left(-\beta(\abs{x}/\sqrt{t})^{2-\delta}\right),\quad\text{where }\;C_1=C_1(d).
\]
On the other hand, in the case when $(t,x) \in (0,1)\times \bR^d$ satisfies $\abs{x}/\sqrt{t} \le R_0$, then we can use \eqref{eq2310w} to bound $\Gamma(t,x,0,0)$.

In conclusion, we have the following:
For any $\delta\in(0,1)$, there exists $C=C(d, \lambda, \Lambda, \omega_{\mathbf A}^{\mathsf x}, \delta)$ such that
\[
\abs{\Gamma(t,x,0,0)} \le  Ct^{-d/2} \exp\left(-\beta (\abs{x}/\sqrt{t})^{2-\delta}\right)\quad\text{on }\;(0,1]\times \bR^d.
\]
Finally, by translation and the semigroup property \eqref{eq1549th}, we get \eqref{eq1742sun}.\qed
\section{Proof of Theorem~\ref{thm3}}			\label{sec5}

Let $\Gamma(t,x, \tau,\xi)$ be the fundamental solution of the operator $P$ constructed in Section~\ref{sec3}.
Let $y \in \bR^n$ be fixed and let $\bar P^{y}$ be given by
\[
\bar P^{y} u = \partial_t u - a^{ij}(t, y) D_{ij}u.
\]
Let $\Phi^{y}(t,x, \tau,\xi)$ be the fundamental solution of the operator $\bar P^{y}$.
Notice that the coefficients of $\bar P^{y}$ depend only on $t$ and thus one can  compute $\Phi^{y}(t,x, \tau,\xi)$ by using the Fourier transform.
However, we do not need its explicit form and will just make use of the following fact.
For $t>\tau$ we have
\begin{equation}				\label{eq1638f}
\begin{aligned}
\abs{\Phi^y (t,x,\tau,\xi)}  &\le \frac{C_0}{(t-\tau)^{d/2}} e^{-\kappa_0 \frac{\abs{x-\xi}^2}{t-\tau}},\\\
\abs{D^2_{x}\Phi^y(t,x,\tau,\xi)}  &\le \frac{C_0'}{(t-\tau)^{d/2}}\left(\frac{1}{t-\tau}+\frac{\abs{x-\xi}^2}{(t-\tau)^2}\right) e^{-\kappa_0 \frac{\abs{x-\xi}^2}{t-\tau}},
\end{aligned}
\end{equation}
where $C_0=C_0(d,\lambda, \Lambda)$, $C_0'=C_0'(d, \lambda, \Lambda)$, and $\kappa_0=\kappa_0(\lambda, \Lambda)$ are positive constants.

\subsection{Modified parametrix method}
Notice that we have
\begin{align*}
P \Gamma(t,x, \tau,\xi)&-P\Phi^{y}(t,x, \tau,\xi) \\
&= P\Gamma(t,x, \tau,\xi)  -P \Phi^{y}(t,x, \tau,\xi) + \bar P^{y} \Phi^{y}(t,x, \tau,\xi) - \bar P^{y} \Phi^{y}(t,x, \tau,\xi)\\
&=-(P-\bar P^{y}) \Phi^{y}(t,x,\tau,\xi)= (a_{ij}(t,x)-a_{ij}(t,y))D_{ij} \Phi^{y}(t,x,\tau,\xi).
\end{align*}
In particular, by taking $y=\xi$ and setting
\begin{equation}				\label{eq1601th}
v(t,x, \tau,\xi):=\Gamma(t,x, \tau,\xi)-\Phi^{\xi}(t,x, \tau,\xi),
\end{equation}
we have
\begin{equation*}				
P v(t,x, \tau,\xi)= (a_{ij}(t,x)-a_{ij}(t,\xi))D_{ij} \Phi^{\xi}(t,x,\tau,\xi).
\end{equation*}
We shall shortly show that the following representation formula is available:
\begin{equation}			\label{eq2244sat}
v(t,x,\tau,\xi)=\int_{\tau}^t \int_{\bR^d} \Gamma(t,x,s,y)  (a_{ij}(s,y)-a_{ij}(s,\xi))D_{ij} \Phi^{\xi} (s,y, \tau,\xi)\,dy ds.
\end{equation}
It then follows from \eqref{eq2244sat} that $v$ satisfies the relation
\begin{align}				
						\nonumber
v(t,x,\tau,\xi)&=\int_{\tau}^t \int_{\bR^d} \Phi^{y}(t,x,s,y)  (a_{ij}(s,y)-a_{ij}(s,\xi))D_{ij} \Phi^\xi(s,y,\tau,\xi)\,dy ds\\
						\label{eq1616th}
&\qquad +\int_{\tau}^t \int_{\bR^d} v(t,x,s,y)  (a_{ij}(s,y)-a_{ij}(s,\xi))D_{ij} \Phi^\xi(s,y, \tau,\xi)\,dy ds.
\end{align}
We note that both integrals in \eqref{eq1616th} are absolutely convergent.
See \eqref{eq2049f}.
The last formula is reminiscent of the classical parametrix method for constructing the fundamental solutions.
First, we set
\begin{equation}			\label{eq2057th}
w_0(t,x,\tau,\xi)=\int_{\tau}^t \int_{\bR^d} \Phi^{y}(t,x,s,y)  (a_{ij}(s,y)-a_{ij}(s,\xi))D_{ij} \Phi^\xi(s,y, \tau,\xi)\,dy ds
\end{equation}
and inductively define for $k=0, 1, 2, \ldots$,
\begin{equation}			\label{eq2058th}
w_{k+1}(t,x,\tau,\xi)=\int_{\tau}^t \int_{\bR^d} w_k(t,x,s,y)  (a_{ij}(s,y)-a_{ij}(s,\xi))D_{ij} \Phi^\xi(s,y, \tau,\xi)\,dy ds.
\end{equation}
Suppose that
\begin{equation}			\label{eq2100th}
w(t,x, \tau, \xi):= \sum_{k=0}^\infty w_k(t,x,\tau,\xi)
\end{equation}
converges uniformly.
Then by summing over $k=0,1,2,\ldots$ in \eqref{eq2058th}, we find
\[
w(t,x, \tau, \xi) =  w_0(t,x,\tau, \xi) + \int_{\tau}^t \int_{\bR^d}  w(t,x,s,y)  (a_{ij}(s,y)-a_{ij}(s,\xi))D_{ij} \Phi^{\xi}(s,y, \tau,\xi)\,dy ds.
\]
Since we also have \eqref{eq1616th}, it is plausible that
\begin{equation}				\label{eq1617th}
v(t,x, \tau, \xi)=w(t,x, \tau, \xi).
\end{equation}
We shall verify \eqref{eq1617th} after we establish the Gaussian estimate for $w$.

\subsection{Gaussian estimate for $w$}
Recall that we assume $\mathbf A$ is uniformly Dini continuous in $x$, that is, 
\begin{equation*}
\varrho_{\mathbf A}^{\mathsf x}(r):=
\sup\,\Set{\abs{\mathbf A(t,x)-\mathbf A(t,y)}: x, y \in \bR^d,\,  t\in \bR,\, \abs{x-y}\le r\,}
\end{equation*}
satisfies the Dini condition
\[
\int_0^1 \frac{\varrho_{\mathbf A}^{\textsf x}(r)}{r}\,dr <+\infty.
\]
It follows from \eqref{eq2057th} and \eqref{eq1638f} that
\begin{align}
							\nonumber
&\abs{w_0(t,x, \tau, \xi)} \le \int_{\tau}^{t} \int_{\bR^d} \abs{\Phi^{y}(t,x,s,y)} \abs{\mathbf{A}(s,y)-\mathbf{A}(s,\xi)} \abs{D^2 \Phi^\xi(s,y,\tau, \xi)}\,dy ds\\
							\label{eq1635f}
&\qquad \le\int_{\tau}^{t} \int_{\bR^d}  \frac{C_0 C_0'e^{-\kappa_0 \frac{\abs{x-y}^2}{t-s}} }{(t-s)^{d/2}(s-\tau)^{d/2}}  \varrho_{\mathbf A}^{\mathsf x}(\abs{y-\xi}) \left(\frac{1}{s-\tau}+\frac{\abs{y-\xi}^2}{(s-\tau)^2}\right)e^{-\kappa_0 \frac{\abs{y-\xi}^2}{s-\tau}} dyds.
\end{align}
Since $\varrho_{\mathbf A}^{\mathsf x}$ is increasing and by the triangle inequality, we have
\[
\varrho_{\mathbf A}^{\mathsf x}(r_1+r_2)\le \varrho_{\mathbf A}^{\mathsf x}(r_1)+\varrho_{\mathbf A}^{\mathsf x}(r_2),\qquad\forall r_1,r_2\ge 0,
\]
it follows that
\begin{equation*}
\frac{\varrho_{\mathbf A}^{\mathsf x}(\abs{y-\xi})}{\abs{y-\xi}} \le 2 \frac{\varrho_{\mathbf A}^{\mathsf x}(\sqrt{s-\tau})}{\sqrt{s-\tau}} \quad \text{for }\;\abs{y-\xi} \ge \sqrt{s-\tau}.
\end{equation*}
Therefore, in the case when $\abs{y-\xi} \ge \sqrt{s-\tau}$, we have
\[
\varrho_{\mathbf A}^{\mathsf x}(\abs{y-\xi}) \left(\frac{1}{s-\tau}+\frac{\abs{y-\xi}^2}{(s-\tau)^2}\right)
\le 2  \frac{\varrho_{\mathbf A}^{\mathsf x}(\sqrt{s-\tau})}{s-\tau}\left( \frac{\abs{y-\xi}^2}{s-\tau} \right)^{\frac12}\left(1+\frac{\abs{y-\xi}^2}{s-\tau} \right).
\]
On the other hand, if $\abs{y-\xi} < \sqrt{s-\tau}$, then we have
\[
\varrho_{\mathbf A}^{\mathsf x}(\abs{y-\xi}) \left(\frac{1}{s-\tau}+\frac{\abs{y-\xi}^2}{(s-\tau)^2}\right) \le \frac{\varrho_{\mathbf A}^{\mathsf x}(\sqrt{s-\tau})}{s-\tau} \left(1+\frac{\abs{y-\xi}^2}{s-\tau}\right).
\]
In both cases, notice that for any $\kappa'_0 \in (0, \kappa_0)$, there is a  constant $C_1=C_1(\kappa_0, \kappa_0')>0$ such that we have
\begin{equation}				\label{eq1541f}
\varrho_{\mathbf A}^{\mathsf x}(\abs{y-\xi}) \left(\frac{1}{s-\tau}+\frac{\abs{y-\xi}^2}{(s-\tau)^2}\right) e^{-\kappa_0 \frac{\abs{y-\xi}^2}{s-\tau}}  \le C_1 \frac{\varrho_{\mathbf A}^{\mathsf x}(\sqrt{s-\tau})}{s-\tau} e^{-\kappa_0' \frac{\abs{y-\xi}^2}{s-\tau}}.
\end{equation}
We recall the following identity, which is a simple consequence of the Fourier transform:
For $\tau < s<t$, we have
\begin{equation}				\label{lem5.9}
\int_{\bR^d}  \frac{1}{(t-s)^{d/2}} e^{-\kappa_0' \frac{\abs{x-y}^2}{t-s}} \frac{1}{(s-\tau)^{d/2}}  e^{-\kappa_0' \frac{\abs{y-\xi}^2}{s-\tau}}\,dy = C_2 \frac{1}{(t-\tau)^{d/2}} e^{-\kappa_0' \frac{\abs{x-\xi}^2}{t-\tau}},
\end{equation}
where
\[
C_2= \int_{\bR^d} e^{-\kappa_0' \abs{y}^2}\,dy= (\pi/\kappa_0')^{d/2}.
\]

Therefore, by plugging in \eqref{eq1541f} into \eqref{eq1635f} and using the identity \eqref{lem5.9}, we get
\begin{align}
							\nonumber
\abs{w_0(t,x, \tau, \xi)} &\le C_0 C_0' C_1 \int_{\tau}^{t} \frac{\varrho_{\mathbf A}^{\mathsf x}(\sqrt{s-\tau})}{s-\tau} \left(\int_{\bR^d}  \frac{1}{(t-s)^{d/2}} e^{-\kappa_0' \frac{\abs{x-y}^2}{t-s}} \frac{1}{(s-\tau)^{d/2}}  e^{-\kappa_0' \frac{\abs{y-\xi}^2}{s-\tau}}\,dy\right) ds\\
							\label{eq1603f}
&\le \left(C_0 C_0' C_1 C_2  \int_{\tau}^{t} \frac{\varrho_{\mathbf A}^{\mathsf x}(\sqrt{s-\tau})}{s-\tau}\,ds\right)  \frac{1}{(t-\tau)^{d/2}} e^{-\kappa_0' \frac{\abs{x-\xi}^2}{t-\tau}}.
\end{align}
Note that
\[
\int_{\tau}^{t} \frac{\varrho_{\mathbf A}^{\mathsf x}(\sqrt{s-\tau})}{s-\tau}\,ds= 2 \int_0^{\sqrt{t-\tau}} \frac{\varrho_{\mathbf A}^{\mathsf x}(s)}{s}\,ds.
\]
Let $\varepsilon_0 \in (0,1)$ be to fixed later.
Take $\delta_0>0$ such that
\begin{equation}							\label{eq1604f}
2C_0' C_1 C_2  \int_0^{\delta_0} \frac{\varrho_{\mathbf A}^{\mathsf x}(s)}{s}\,ds \le \varepsilon_0.
\end{equation}
Then we find from \eqref{eq1603f} and \eqref{eq1604f} that
\begin{equation}							\label{eq1605f}
\abs{w_0(t,x, \tau, \xi)} \le \varepsilon_0 C_0 \frac{1}{(t-\tau)^{d/2}} e^{-\kappa_0' \frac{\abs{x-\xi}^2}{t-\tau}}\quad\text{provided } 0< t-\tau \le \delta_0^2.
\end{equation}

Now using \eqref{eq2058th}, \eqref{eq1605f}, and \eqref{eq1638f}, we get
\begin{align*}
&\abs{w_1(t,x, \tau, \xi)} \le \int_{\tau}^{t} \int_{\bR^d} \abs{w_0(t,x,s,y)} \abs{\mathbf{A}(s,y)-\mathbf{A}(s,\xi)} \abs{D^2 \Phi^\xi(s,y,\tau, \xi)}\,dy ds\\
&\qquad\le\varepsilon_0  \int_{\tau}^{t} \int_{\bR^d}  \frac{C_0 C_0'e^{-\kappa_0' \frac{\abs{x-y}^2}{t-s}} }{(t-s)^{d/2}(s-\tau)^{d/2}}  \varrho_{\mathbf A}^{\mathsf x}(\abs{y-\xi}) \left(\frac{1}{s-\tau}+\frac{\abs{y-\xi}^2}{(s-\tau)^2}\right)e^{-\kappa_0 \frac{\abs{y-\xi}^2}{s-\tau}} dyds.
\end{align*}
By using \eqref{eq1541f} and repeating the same computation as in \eqref{eq1603f}, we get
\begin{equation*}							
\abs{w_1(t,x, \tau, \xi)} \le \varepsilon_0^2 C_0 \frac{1}{(t-\tau)^{d/2}} e^{-\kappa_0' \frac{\abs{x-\xi}^2}{t-\tau}}\quad\text{provided } 0< t-\tau \le \delta_0^2.
\end{equation*}
Inductively, we have
\begin{equation*}							
\abs{w_k(t,x, \tau, \xi)} \le \varepsilon_0^{k+1} C_0 \frac{1}{(t-\tau)^{d/2}} e^{-\kappa_0' \frac{\abs{x-\xi}^2}{t-\tau}}\quad\text{provided } 0< t-\tau \le \delta_0^2.
\end{equation*}
Then by \eqref{eq2100th}, we have for $0< t-\tau \le \delta_0^2$ that
\begin{equation}							\label{eq1657f}
\abs{w(t,x, \tau, \xi)} \le \sum_{k=0}^\infty \varepsilon_0^{k+1} C_0 \frac{1}{(t-\tau)^{d/2}} e^{-\kappa_0' \frac{\abs{x-\xi}^2}{t-\tau}} \le \frac{\varepsilon_0 C_0}{1-\varepsilon_0} \frac{1}{(t-\tau)^{d/2}} e^{-\kappa_0' \frac{\abs{x-\xi}^2}{t-\tau}}.
\end{equation}

\subsection{Verification of \eqref{eq2244sat} and \eqref{eq1617th}}
We shall prove \eqref{eq2244sat} first.
Let us denote
\[
f(s,y):=(a_{ij}(s,y)-a_{ij}(s,\xi))D_{ij} \Phi^\xi(s,y, \tau,\xi).
\]
Notice that in deriving \eqref{eq1603f}, we have seen that
\begin{equation}			\label{eq1407sun}
\abs{f(s,y)} \le C\, \frac{\varrho_{\mathbf A}^{\mathsf x}(\sqrt{s-\tau})}{s-\tau} \frac{1}{(s-\tau)^{d/2}} e^{-\kappa_0' \frac{\abs{y-\xi}^2}{s-\tau}}.
\end{equation}
Write $Z=(\tau,\xi)$ and let $\zeta$ be a smooth function on $\bR^{d+1}$ such that
\begin{equation*}			
0\leq \zeta \leq 1, \;\; \zeta=0 \;\text{ in }\;Q_{r/2}(Z),\;\; \zeta=1\;\text{ in }\bR^{d+1}\setminus Q_{r}(Z),\;\; \abs{\partial_t \zeta}+ \abs{D\zeta}^2+\abs{D^2\zeta}\le C r^{-2},
\end{equation*}
where, $0<r < \frac14 (t-\tau)$.
Then, $\tilde v=\zeta v(\cdot,\cdot, \tau, \xi)$ satisfies
\[
P\tilde v=\zeta f+ v P\zeta -2a^{ij}D_i v D_j \zeta\;\text{ in }\;(\tau, t)\times \bR^d,\quad \tilde v(\tau, \cdot)=0 \;\text{ on }\;\bR^d.
\]
Notice that  $\zeta f+ v P\zeta -2a^{ij}D_i v D_j \zeta \in L_{p}((\tau, t)\times \bR^d)$ with $p>(d+2)/2$.
Therefore, by Proposition~\ref{prop3.5} and the symmetry relation \eqref{symmetry},  we have
\begin{align*}
&v(t, x, \tau,\xi) = \tilde v(t,x) = I+II:= \int_{\tau}^{t} \int_{\bR^d} \Gamma(t,x,s,y) \zeta(s,y) f(s,y)\,dyds\\
&\quad+\int_{Q_{r}(Z) \setminus Q_{r/2}(Z)}\!\!\! \Gamma(t,x,s,y) \left\{P\zeta(s,y)v(s,y,\tau,\xi) -2a^{ij}D_i \zeta(s,y) D_j v(s,y,\tau,\xi) \right\}\,dyds.
\end{align*}
We claim that $II \to 0$ as $r \to 0$.
Assume the claim for now.
By \eqref{eq1407sun} and \eqref{eq1742sun}, we see that $I$ is absolutely convergent, that is,
\begin{multline}				\label{eq2049f}
\int_\tau^t \int_{\bR^d}  \abs{\Gamma(t,x,s,y) f(s,y)} \,dyds\\
=\int_\tau^{\frac{t+\tau}{2}}\!\!\! \int_{\bR^d}  + \int_{\frac{t+\tau}{2}}^t \int_{\bR^d} \abs{\Gamma(t,x,s,y) f(s,y)} \,dyds <+\infty,
\end{multline}
and thus we obtain \eqref{eq2244sat} by the dominated convergence theorem applied to
$I$.

Now, we prove the claim that $II \to 0$.
For $Y=(s,y) \in (\tau,t)\times \bR^d$, let
\[
\tilde r=\tfrac15 \abs{Y-Z}.
\]
We set $\delta=\alpha \tilde r$, where $\alpha>1$ is to be specified.
Recall that
\[
v(s,y,\tau,\xi)=\Gamma(s,y,\tau, \xi) - \Phi^{\xi}(s,y,\tau, \xi)=\Gamma^\ast(\tau, \xi, s, y)-(\Phi^{\xi})^\ast(\tau, \xi, s,y),
\]
and note that
\begin{equation}				\label{eq2054w}
v^\ast=v^\ast(\cdot, \cdot)=\Gamma(s,y,\cdot,\cdot)-\Phi^{\xi}(s,y,\cdot,\cdot)
\end{equation}
satisfies
\[
P^\ast v^\ast = \dv^2((\mathbf A-\widetilde{\mathbf A}_0) \widetilde \Phi_0)\;\text{ in }\; (\tau -1,s) \times \bR^d,\quad v^\ast(s,\cdot)=0 \;\text{ on }\;\bR^d,
\]
where we set
\[
\widetilde{\mathbf A}_0=\widetilde{\mathbf A}_0(\cdot)=\mathbf A(\cdot ,\xi)\quad\text{and}\quad \widetilde \Phi_0=\widetilde \Phi_0(\cdot,\cdot)= (\Phi^{\xi})^\ast(\cdot, \cdot, s, y).
\]
Let $\tilde\zeta$ be a smooth function on $\bR^{d+1}$ such that
\[
0\le \tilde\zeta \leq 1, \quad \tilde \zeta=0 \;\text{ in }\;Q_{\delta/2}(Y),\quad \tilde\zeta=1\;\text{ in }\bR^{d+1}\setminus Q_{\delta}(Y),\quad \abs{D\tilde \zeta} \le 4/\delta,
\]
and define $\mathbf{g}_1$ and $\mathbf{g}_2$ by (cf. \eqref{eq1654t})
\begin{equation*}
\mathbf{g}_1= \tilde\zeta(\mathbf{A}- \widetilde{\mathbf A}_0) \widetilde \Phi_0 \quad\text{and}\quad \mathbf{g}_2= (1-\tilde\zeta)(\mathbf{A}-\widetilde{\mathbf A}_0) \widetilde \Phi_0.
\end{equation*}
Noting that $\norm{\mathbf A-\widetilde{\mathbf A}_0}_\infty \le C(d,\Lambda)$ and using \eqref{eq1948sun} and properties of $\tilde \zeta$, we have
\begin{equation}
					\label{eq1730f}
\int_{\bR^{d+1}} \abs{\mathbf g_1}^{\frac{2(d+2)}{d}} \le \sum_{k=0}^\infty \int_{Q_{2^{k}\delta}(Y)\setminus Q_{2^{k-1}\delta}(Y)} \abs{\mathbf{g}_1}^{\frac{2(d+2)}{d}}
\le C \sum_{k=0}^\infty (2^k \delta)^{-d-2}  \le C\delta^{-d-2}.
\end{equation}
Note that we have
\[
\abs{\mathbf A-\widetilde{\mathbf A}_0} \le \varrho_{\mathbf A}^{\mathsf x}(\delta) + \varrho_{\mathbf A}^{\mathsf x}(5\tilde r) \le (\alpha+6)\varrho_{\mathbf A}^{\mathsf x}(\tilde r) \quad \text{in }\; Q_{\delta}(Y),
\]
and thus we have
\begin{equation}
						\label{eq1432f}
\int_{\bR^{d+1}}\, \abs{\mathbf g_2}^{\frac{d+1}{d}} \le \left((\alpha+6) \varrho_{\mathbf A}^{\mathsf x}(\tilde r)\right)^{\frac{d+1}{d}}  \int_{Q_{\delta}(Y)} \abs{\mathbf{g}_2}^{\frac{d+1}{d}}
 \le  C\left((\alpha+6) \varrho_{\mathbf A}^{\mathsf x}(\tilde r)\right)^{\frac{d+1}{d}} \delta.
\end{equation}
Let $v_i$ ($i=1,2$) be the solutions of the problems
\[
P^\ast v_i = \dv^2 \mathbf{g}_i\;\text{ in }\; (\tau-1,s)\times \bR^d,\quad
v_i(s,\cdot)=0\;\text{ on }\;\bR^d \qquad(i=1,2).
\]
We extend $v_1$ and $v_2$ to be zero on $(s,\infty)\times \bR^d$.
By Lemma~\ref{lem01} together with \eqref{eq1730f} and \eqref{eq1432f}, we have
\begin{equation}				\label{eq1618f}
\norm{v_1}_{L_{2(d+2)/d}((\tau-1,s)\times \bR^d)} \le C\delta^{-\frac{d}{2}}\;\text{ and }\;
\norm{v_2}_{L_{(d+1)/d}((\tau-1,s)\times \bR^d)} \le C (\alpha+6) \varrho_{\mathbf A}^{\mathsf x}(\tilde r)\delta^{\frac{d}{d+1}}.
\end{equation}
By \eqref{eq0854m}, we see that both $v_1$ and $v_2$ also belong to $L_{p}((\tau-1,s)\times \bR^d)$ for any $p\in (1, (d+2)/d)$.
Therefore, by the uniqueness, we have
\begin{equation}			\label{eq1710m}
v^\ast =v_1+v_2.
\end{equation}
We now estimate $v_1(Z)$ and $v_2(Z)$. By using  Lemma~\ref{lem02}, we have
\begin{equation}				\label{eq1335sat}
\abs{v_i(Z)} \le  C \fint_{Q_{2\tilde r}^{+}(Z)} \abs{v_i} + C \int_0^{\tilde r} \frac{\omega_{\mathbf{g}_i}^{\mathsf x}(\tilde t, Q_{2\tilde r}^{+}(Z))}{\tilde t}\,d\tilde t\qquad (i=1,2).
\end{equation}
Using \eqref{eq1618f} together with H\"older's inequalities, we have
\begin{equation}				\label{eq1340sat}
\begin{aligned}
\fint_{Q_{2\tilde r}^{+}(Z)} \abs{v_1} &\le C {\tilde r}^{-\frac{d}{2}}  \norm{v_1}_{L_{2(d+2)/d}(Q_{2\tilde r}^{+}(Z))} \le C  {\tilde r}^{-\frac{d}{2}}\delta^{-\frac{d}{2}},\\
\fint_{Q_{2\tilde r}^{+}(Z)} \abs{v_2} &\le C{\tilde r}^{-\frac{(d+2)d}{d+1}}  \norm{v_2}_{L_{(d+1)/d}Q_{2\tilde r}^{+}(Z))} \le C{\tilde r}^{-d} (\alpha+6) \alpha^{\frac{d}{d+1}} \varrho_{\mathbf A}^{\mathsf x}(\tilde r).
\end{aligned}
\end{equation}
By using the bound of $\widetilde \Phi_0$, we have
\begin{equation}				\label{eq1935w}
\omega_{\mathbf g_i}^{\mathsf x}(\tilde t, Q_{2\tilde r}^+(Z)) \le C(d,\lambda,\Lambda) {\tilde r}^{-d} \varrho_{\mathbf A}^{\mathsf x}(\tilde t),\quad \forall \tilde t \in (0, \tilde r]\qquad (i=1,2).
\end{equation}
By combining \eqref{eq1710m}, \eqref{eq1335sat}, \eqref{eq1340sat}, and \eqref{eq1935w}, we obtain
\[
\abs{v^\ast(Z)}
\le C {\tilde r}^{-d}\left(\alpha^{-\frac{d}{2}}+ (\alpha+6) \alpha^{\frac{d}{d+1}} \varrho_{\mathbf A}^{\mathsf x}(\tilde r) +\int_0^{\tilde r} \frac{\varrho_{\mathbf A}^{\mathsf x}(\tilde t)}{\tilde t}\,d\tilde t\right),
\]
where $C$ is a constant independent of $\tilde r$.
Recall that $\abs{Y-Z}=5\tilde r$.
Now for any $\varepsilon\in (0,1)$, we can take $\alpha>1$ sufficiently large and then $\tilde r$ sufficiently small such that
\begin{equation}				\label{eq1406sat}
\abs{v^\ast(Z)} \le \varepsilon \abs{Y-Z}^{-d}.
\end{equation}
Therefore, we conclude from \eqref{eq1406sat} and \eqref{eq2054w} that for all small $r>0$,
\[
r^d \abs{v(s,y,\tau,\xi)} =o(r), \quad \forall (s,y) \in Q_{2r}(Z)\setminus Q_{r/4}(Z),
\]
where we use $o(r)$ to denote some bounded quantity that tends to $0$ as $r \to 0$.

To estimate $Dv$, we use the equation $Pv=f$. Notice that
\[
\abs{f(s,y)} \le C \varrho_{\mathbf A}^{\mathsf x}(\abs{y-\xi}) \abs{Z-Y}^{-d-2} \le C  \varrho_{\mathbf A}^{\mathsf x}(r) r^{-d-2} \quad\text{in }\;Q_{2r}(Z) \setminus Q_{r/4}(Z).
\]
Therefore, by using \eqref{eq1406sat}, the local $W^{1,2}_p$ estimate
\begin{align*}
&\Norm{r^2 \abs{\partial_s v(\cdot,\cdot,\tau,\xi)}+
r^2 \abs{D_y^2v(\cdot,\cdot,\tau,\xi)}+r\abs{D_y v(\cdot,\cdot,\tau,\xi)}}_{L_p(Q_{r}(Z) \setminus Q_{r/2}(Z))}\\
&\qquad\qquad \le C \norm{v(\cdot,\cdot,\tau,\xi)}_{L_p(Q_{2r}(Z) \setminus Q_{r/4}(Z))}
+Cr^2 \norm{f}_{L_p(Q_{2r}(Z) \setminus Q_{r/4}(Z))},\quad p>d+2,
\end{align*}
and the Sobolev embedding, we have
\begin{equation}				\label{eq1940m}
r^{d+1} \abs{Dv(s,y,\tau,\xi)} =o(r),\quad  \forall (s,y) \in Q_{r}(Z)\setminus Q_{r/2}(Z).
\end{equation}
Therefore, by using \eqref{eq1406sat}, \eqref{eq1940m}, and the properties of $\zeta$, we get $II \to 0$ as $r \to 0$, which completes the proof of \eqref{eq2244sat}.

To show \eqref{eq1617th}, we invoke the contraction mapping theorem.
For $(t,x) \in \bR^{d+1}$, let $\mathscr B=L_1((t-\delta_0^2, t)\times \bR^d)$, where $\delta_0$ is as in \eqref{eq1604f}.
We shall show that the mapping $T: \mathscr B \to \mathscr B$ defined by
\[
T u(\tau, \xi) =  w_0(t,x, \tau, \xi) + \int_{\tau}^t \int_{\bR^d}  u(s,y)  (a_{ij}(s,y)-a_{ij}(s,\xi))D_{ij} \Phi^{\xi}(s,y, \tau,\xi)\,dy ds
\]
is a contraction.
Indeed, by \eqref{eq1605f}, \eqref{eq1541f}, Fubini's theorem, and \eqref{eq1604f}, we find that
\begin{align*}
&\int_{t-\delta_0^2}^t \int_{\bR^d} \abs{Tu(\tau, \xi)}\,d\xi d\tau \le  \int_{t-\delta_0^2}^t \int_{\bR^d} \abs{w_0(t,x, \tau, \xi)} \,d\xi d\tau\\
&\quad+ \int_{t-\delta_0^2}^t \int_{\bR^d} \int_\tau^t \int_{\bR^d} \abs{u(s,y)} C_0'\frac{ \varrho_{\mathbf A}^{\mathsf x}(\abs{y-\xi})}{(s-\tau)^{d/2}} \left(\frac{1}{s-\tau}+\frac{\abs{y-\xi}^2}{(s-\tau)^2}\right)  e^{-\kappa_0 \frac{\abs{y-\xi}^2}{s-\tau}}\,dy ds d\xi d\tau\\
&\le  \int_{t-\delta_0^2}^t \int_{\bR^d}\varepsilon_0 C_0  \frac{1}{(t-\tau)^{d/2}} e^{-\kappa_0' \frac{\abs{x-\xi}^2}{t-\tau}}\,d\xi d\tau\\
&\quad + \int_{t-\delta_0^2}^t \int_{\bR^d}\abs{u(s,y)} \int_{t-\delta_0^2}^s   \frac{\varrho_{\mathbf A}^{\mathsf x}(\sqrt{s-\tau})}{s-\tau} \int_{\bR^d} \frac{C_0' C_1}{(s-\tau)^{d/2}}  e^{-\kappa_0' \frac{\abs{y-\xi}^2}{s-\tau}}\,d\xi d\tau dy ds\\
&\le  \varepsilon_0 C_0 C_2\delta_0^2 + \varepsilon_0  \int_{t-\delta_0}^t \int_{\bR^d}\abs{u(s,y)}\,dyds.
\end{align*}
Therefore, we have $T u \in \mathscr B$ for all $u \in \mathscr B$.
By a similar calculation, we also find that
\[
\norm{Tu_1 - Tu_2}_{\mathscr B} \le \varepsilon_0 \norm{u_1-u_2}_{\mathscr B},
\]
which implies $T$ is a contraction mapping on $\mathscr B$ since we assume $\varepsilon_0 \in (0,1)$.
We now fix $\varepsilon_0 =1/2$.
Note that it follows from \eqref{eq1742sun} and \eqref{eq1657f}, respectively, that $v \in \mathscr B$ and $w \in \mathscr B$, which establishes the equality \eqref{eq1617th}.

\subsection{Conclusion}
Therefore, by \eqref{eq1601th} \eqref{eq1617th}, \eqref{eq1657f}, and \eqref{eq1638f}, we find that
\begin{equation}							\label{eq1721f}
\abs{\Gamma(t,x, s, y)} \le  \frac{C}{(t-s)^{d/2}} e^{-\kappa_0' \frac{\abs{x-y}^2}{t-s}}\quad\text{provided } 0< t-s \le \delta_0^2.
\end{equation}
We can take $\kappa_0'=\kappa_0/2$ in the above and call it $\kappa$. It is clear that $\kappa$ then depends only on $\lambda$ and $\Lambda$.
By using \eqref{eq1549th} and \eqref{eq1721f}, we establish the Gaussian bound \eqref{eq1743sun}.
See e.g., \cite[Section~5.5]{CDK08} for the details. \qed


\end{document}